\documentclass[12pt]{article}
\usepackage{color,bbold}%, hyperref
\usepackage{amssymb,amsmath,amsthm,graphicx, color, amsfonts,tabularx}
\usepackage[english]{babel}
\usepackage{afterpage}
\usepackage{authblk, relsize}
\usepackage{float}%insert images in the exact place
\usepackage{geometry}
\def\supp{\mathop{\rm supp}\nolimits}

\newtheorem{theorem}{Theorem}[section]

\newtheorem{lemma}[theorem]{Lemma}
\newtheorem{proposition}[theorem]{Proposition}
\newtheorem{corollary}[theorem]{Corollary}
\newtheorem{definition}[theorem]{Definition}
\newtheorem{remark}[theorem]{Remark}

\newtheorem{notation}[theorem]{Notation}

\newcommand{\N}{\mathbb{N}}

\newcommand{\R}{\mathbb{R}}

\renewenvironment{proof}[1][.]{%
\bigskip\noindent{\bf Proof#1 }}{%
\hfill$\blacksquare$\bigskip}

\newcommand{\on}{\operatorname}
\newcommand{\vp}{\varphi}

\begin{document}
\pagestyle{myheadings}
\title{Invariant measures for place dependent idempotent iterated function systems}
\author[1]{Jairo K. Mengue
\thanks{E-mail: jairo.mengue@ufrgs.br}}
\author[2]{Elismar R. Oliveira
\thanks{E-mail: elismar.oliveira@ufrgs.br}}
\affil[1,2]{Universidade Federal do Rio Grande do Sul}

\date{\today}
\maketitle

\begin{abstract}
We study the set of invariant idempotent probabilities for place dependent idempotent iterated function systems defined in compact metric spaces. Using well-known ideas from dynamical systems, such as the Ma\~{n}\'{e} potential and the Aubry set, we provide a complete characterization of the densities of such  idempotent probabilities. As an application, we provide an alternative formula for the attractor of a class of fuzzy iterated function systems.
\end{abstract}
\vspace {.8cm}



\section{Introduction}
The theory of idempotent probabilities was introduced by Maslov (also called Maslov measures) in \cite{Kol97} and \cite{LitMas96} to model problems of optimization such as  the Hamilton-Jacob equation by considering those equations as integrals with respect to a max-plus structure (sometimes min-plus, max-min, etc).

In the study of dynamical systems, several objects obtained in ergodic theory has counterparts as idempotent probabilities when considered  as  zero temperature limits. This is the case, for example, of the main eigenfunction for the transfer operator which corresponds to calibrated subactions \cite{CG}, \cite{BLT} and \cite{BLL}.

In the present paper our focus is the description of the invariant idempotent probabilities for a place dependent  transfer operator associated to an iterated function system (IFS). In the non place dependent case, there exists a unique invariant idempotent probability, see \cite{MZ}. Later, in \cite{dCOS20fuzzy} and \cite{ExistInvIdempotent} it was shown  that this unique measure can be obtained by iteration of a contractive operator, providing, in this way, a characterization via limit with respect to an appropriate metric.

We will prove that in the place dependent case it is possible to exhibit examples with infinitely many invariant idempotent probabilities and our main result (see Theorem \ref{teo: main result}) presents a characterization of the associated densities. Such characterization uses two important concepts of ergodic optimization, Ma\~{n}\'{e} potential and Aubry set. It is a very interesting fact that in ergodic optimization a similar characterization for subactions is well known, but it is obtained with a different analysis because it is given in the opposite order of variables (compare Theorem \ref{teo: main result} with \cite[Lemma 14]{LM} - see also  \cite[Theorem 4.7]{BLL} and   \cite[Theorem 10]{GaribLop08}).

Let us consider the  max-plus semiring $\mathbb{R}_{\max}:=\mathbb{R}\cup \{-\infty\}$ endowed with the operations
\begin{enumerate}
	\item $\oplus: \mathbb{R}_{\max} \times \mathbb{R}_{\max} \to \mathbb{R}_{\max}$, where $a \oplus b :=\max(a,b)$ assuming $a \oplus -\infty:=a$. The max-plus \emph{additive} neutral element is $\mathbb{0} :=-\infty$.
	\item $\odot: \mathbb{R}_{\max} \times \mathbb{R}_{\max} \to \mathbb{R}_{\max}$, where $a \odot b :=a+b$ assuming $a \odot -\infty:=-\infty$. The max-plus \emph{multiplicative} neutral element is $\mathbb{1} :=0$.
\end{enumerate}

We consider also a compact metric space $(X,d)$ and $C(X,\mathbb{R})$, the set of continuous functions from $X$ to $\mathbb{R}$.
We notice that $\mathcal{V}:=(C(X,\mathbb{R}), \oplus, \odot)$ has a natural $\mathbb{R}$-semimoduli (a vectorial space over a semiring) structure:
\begin{enumerate}
	\item $(a \odot f)(x):= a \odot f(x)$, for $a \in \mathbb{R}$ and $f\in C(X,\mathbb{R})$;
	\item $(f \oplus g)(x):=f (x)\oplus g(x)$ for $f, g \in C(X,\mathbb{R})$.
\end{enumerate}

A function $m: C(X,\mathbb{R}) \to \mathbb{R}$ is a max-plus linear  functional if
	\begin{enumerate}
		\item  $m(a \odot f)=a \odot m( f)$, $\forall a \in \mathbb{R}$ and $\forall f \in C(X,\mathbb{R})$(max-plus homogeneity);
		\item  $m(f \oplus g)=m( f) \oplus m( g)$, $\forall f, g \in C(X,\mathbb{R})$ (max-plus additive).
	\end{enumerate}

	We denote by $C^{*}(X,\mathbb{R})$ the max-plus dual of $C(X,\mathbb{R})$, which is the set of all max-plus linear functionals $m: C(X,\mathbb{R}) \to \mathbb{R}$. An element $m \in C^{*}(X,\mathbb{R})$ is called a \emph{Maslov measure} or an \emph{idempotent measure} on $X$.

\begin{definition}\label{def: Idempotent probability as a functional}
	We define $I(X)$ as the subset of  $C^{*}(X,\mathbb{R})$ of all Maslov measures satisfying  $m(0)=0$. An element $m \in I(X)$ is called a \emph{Maslov probability} or an \emph{idempotent probability} on $X$.
\end{definition}
 Using the notation $\mathbb{1}=0$, a Maslov probability satisfies $m(\mathbb{1})=\mathbb{1}$. We notice that $\mu(c)=c $, for all $c \in\mathbb{R}$ and $\mu\in I(X)$.  Another consequence of the definition is that an idempotent probability  is an order-preserving functional, that is, if $\varphi \leq \psi$ then  $\mu(\varphi) \leq \mu(\psi)$, for $\varphi, \psi \in C(X, \mathbb{R})$.

As it is well known, the idempotent measures are closely related with the upper semi-continuous  functions (u.s.c.) (see Definition \ref{def:usc}). The support of a u.s.c. function $\lambda:X\to \mathbb{R}_{\max}$ is the closed set $$\supp( \lambda):=\{x \in X | \lambda(x) \neq -\infty\}.$$ From now on we denote by $U(X, \mathbb{R}_{\max})$  the set of u.s.c. functions $\lambda:X\to \mathbb{R}_{\max}$ with $\supp( \lambda) \neq \varnothing$. We also denote  $\bigoplus_{x \in X} := \sup_{x \in X}$. As $X$ is compact, any function $\lambda\in U(X, \mathbb{R}_{\max})$ attains its supremum.

 The next result is well known in the literature and it plays a central role in this paper, but as far as we know it was proved in \cite{KM89} under a different context. We propose to present a proof for it in the Appendix section~\ref{sec:Fundamentals of idempotent analysis}.

\begin{theorem}\label{teo : densidade Maslov} $\mu:C(X, \mathbb{R}) \to  \mathbb{R} $ is an idempotent measure if and only if there exists $\lambda\in U(X, \mathbb{R}_{\max})$ satisfying
	\[\mu(\psi) = \bigoplus_{x\in X} \lambda(x)\odot \psi(x),\,\,\,\,\forall \,\psi \in C(X,\mathbb{R}).\]
There is a unique such function $\lambda$ in $U(X, \mathbb{R}_{\max})$  and  $\mu\in I(x)$ if and only if $\oplus_{x\in X}\lambda(x) = 0.$
\end{theorem}

The unique upper semi-continuous function $\lambda$ presented in Theorem \ref{teo : densidade Maslov} will be called the \textit{density} of $\mu$. We will also use the notation $\mu=\bigoplus_{x\in X}\lambda(x)\odot\delta_x\in C^*(X,\mathbb{R})$, where $\delta_x(\psi) = \psi(x)$. In this way, by \emph{support} of an idempotent measure $\mu=\bigoplus_{x \in X}  \lambda(x)\odot \delta_{x}\in X$, we will mean the support of its density.

\begin{definition}\label{def:ucIFS}
Let $(X,d_X)$ and $(J,d_J)$ be compact metric spaces. A \emph{uniformly contractible iterated function system}  $(X,(\phi_j)_{j\in J})$ is a family of maps $\{\phi_j:X\to X\,|\,j\in J\}$  satisfying:  there exists $0<\gamma<1$ such that
\begin{equation}\label{eq: gamma contraction}
	d_X(\phi_{j_1}(x_1),\phi_{j_2}(x_2)) \leq \gamma\cdot[d_J(j_1,j_2)+ d_X(x_1,x_2)],\,\,\forall j_1,j_2\in J, \forall x_1,\,x_2\in X.
	\end{equation}
\end{definition}

From now on we consider also the map $\phi:J\times X\to X$ defined by $\phi(j,x) =\phi_j(x)$. From \eqref{eq: gamma contraction} the map $\phi$ is continuous.

\begin{definition} \label{def:mpIFS with measures}
Let $(X,d_X)$ and $(J,d_J)$ be compact metric spaces.\,	A max-plus IFS (mpIFS for short) $\mathcal{S}=(X, (\phi_j)_{j\in J}, (q_j)_{j\in J})$ is a uniformly contractive IFS $(X, (\phi_j)_{j \in J})$ endowed with a normalized family of weights $(q_j)_{j\in J}$, which is a family of functions $\{q_j:X\to\mathbb{R}\,|\,j\in J\}$ satisfying: \newline
	i. \, there exists a constant $C>0$ such that
	\begin{equation}\label{eq: q lipschitz}
		|q_{j}(x_1)-q_{j}(x_2)|\leq C \cdot d_X(x_1,x_2),\,\,\forall j\in J, \forall x_1,\,x_2\in X;
	\end{equation}
	ii. \,
	\begin{equation}\label{eq: q normalized}
		\bigoplus_{j\in J} q_j(x) = 0, \,\,\forall x\in X;
	\end{equation}
	iii. \, the function $q:J\times X\to\mathbb{R}$, defined by $q(j,x)=q_j(x)$, is continuous.
\end{definition}

We also use the compact notation $\mathcal{S}=(X, \phi, q)$ to denote a mpIFS.

\begin{definition}\label{def:mp ruelle operator}
	To each mpIFS $\mathcal{S}=(X, \phi, q)$ we assign the max-plus transfer operator  $\mathcal{L}_{\phi,q}: C(X, \mathbb{R})   \to C(X, \mathbb{R})  $, defined by
	\begin{equation}\label{eq:mp ruelle operator}
		\mathcal{L}_{\phi,q}(f)(x):=\bigoplus_{j \in J} q_{j}(x)\odot f(\phi_{j}(x)),
	\end{equation}
	for any $f \in C(X, \mathbb{R})$. An idempotent probability  $\mu \in {I}(X)$ is called invariant (with respect to the mpIFS) if $\mu(\mathcal{L}_{\phi,q}(f))=\mu(f)$,  for any $f \in C(X, \mathbb{R})  $.
\end{definition}

In some sense $\mathcal{L}_{\phi,q}(f)(x)$ plays the role of the transfer operator in thermodynamic formalism for IFS (see  \cite{FanRuelle}, \cite{MO}). If $J$ is a finite set, $(X,(\phi_j)_{j\in J})$ is a uniformly contractive IFS and $(p_j(x))_{j\in J}$ are non-negative numbers satisfying $\sum_{j\in J}p_j(x) =1,\, \forall x\in X$, then the transfer operator in thermodynamic formalism for IFS is given by
$$\mathcal{L}_{\phi,p}^{{\rm prob}}(f)(x):=\sum_{j \in J} p_{j}(x) \cdot f(\phi_{j}(x)),$$
for any $f \in C(X,\mathbb{R})$. A Borel probability $\mu$ on $X$ is invariant for $\mathcal{L}_{\phi,p}^{{\rm prob}}$ if  $\mu(\mathcal{L}_{\phi,p}^{{\rm prob}}(f))=\mu(f)$,  for any $f \in C(X, \mathbb{R})$. A correspondence between this two approaches can be achieved when considered large deviations for zero temperature limits  \cite{Aki99}, \cite{BLT} and \cite{M}. We intend to consider such question and then to explain more details in a future paper.

In \cite{MZ}, \cite{dCOS20fuzzy} and \cite{ExistInvIdempotent} was proved that when $J$ is a finite set and $q_j(x)=q_j$ does not depend of $x$, there exists a unique invariant idempotent probability $\mu$ for  $\mathcal{L}_{\phi,q}$, which is attractive. On the other hand, in subsection \ref{Non uniqueness of the invariant idempotent measure for mpIS with variable weights} we will present an example of place dependent weight $q_j(x)$ for a finite set $J$ where there exist infinitely many invariant probabilities for  $\mathcal{L}_{\phi,q}$. In subsection \ref{sec:Representation of max-plus  invariant densities} we present the main result of this work which is a characterization of the invariant idempotent probabilities for  $\mathcal{L}_{\phi,q}$ using tools of ergodic optimization as Ma\~{n}\'{e} potential and Aubry set (see Theorem \ref{teo: main result}). As an application, in Section~\ref{sec: non dependence} we prove the uniqueness and exhibit a characterization of the invariant idempotent probability in the case where $q_j(x)=q_j$ does not depend of $x$ (we assume that $J$ is a compact metric space instead a finite set). Furthermore, in Section~\ref{sec:Applications to fuzzy IFS} we present an application of our results to fuzzy IFSs obtained via conjugation from mpIFSs. Finally, in the Appendix Section \ref{sec:Fundamentals of idempotent analysis} we exhibit a proof of Theorem~\ref{teo : densidade Maslov}, following ideas from ~\cite{KM89}.

\section{Max-plus IFSs on compact metric spaces with a compact set of maps}\label{sec:IFSs on compact spaces with a compact set of maps}

 In the initial part of this section we introduce two operators: $M_{\phi,q}$ acting on idempotent probabilities and $L_{\phi,q}$ acting on densities (see Definition \ref{def:Markov operator} and Definition \ref{def:transfer operator}). The main result of this section is Theorem \ref{thm: equivalences} which explains the equivalence between $\mathcal{L}_{\phi,q}$ and these operators.

   Firstly, we propose to prove that $\mathcal{L}_{\phi,q}$ is well defined.

\begin{proposition}\label{prop: max-plus L preserv} Let $\mathcal{S}=(X, \phi, q)$ be a mpIFS. If $f\in C(X,\mathbb{R})$ then $\mathcal{L}_{\phi,q}(f)\in C(X,\mathbb{R})$.
\end{proposition} 	
\begin{proof}
Let $F:J\times X \to \mathbb{R}$ be given by $F(j,x) =q_{j}(x)\odot f(\phi_{j}(x))$. As $\phi,\,q$ and $f$ are continuous, the function $F$ is continuous. As $J$ is compact and $F$ is continuous we have that $\displaystyle{\mathcal{L}_{\phi,q}(f)(x) = \bigoplus_{j\in J} F(j,x)}$ is continuous.
\end{proof}

\begin{definition}\label{def:Markov operator}
	To each mpIFS $\mathcal{S}=(X, \phi, q)$ we assign an operator  $M_{\phi,q}: I(X) \to I(X)$ defined by
	\begin{equation}\label{eq:Markov operator}
		M_{\phi,q}(\mu):= \bigoplus_{j \in J} I_{j}(\mu)
	\end{equation}
	where, $I_{j}(\mu)(f):=\mu( q_{j}\odot (f\circ \phi_{j}))$, for any $f \in C(X, \mathbb{R})  $. An idempotent probability  $\mu \in I(X)$ is called invariant (with respect to the mpIFS) if $M_{\phi,q}(\mu)=\mu$.
\end{definition}

The next  proposition shows that $M_{\phi,q}$ is well defined.

\begin{proposition} If $\mu \in I(X)$ then $M_{\phi,q}(\mu)\in I(X)$.
\end{proposition}
\begin{proof} We start by proving that $M_{\phi,q}(\mu) \in C^*(X,\mathbb{R})$. As $\mu\in I(x)$ it is easy to conclude that
$I_j(\mu)(f\oplus g) = 	I_{j}(\mu)(f)\oplus I_{j}(\mu)( g)$ and $I_{j}(\mu)(c\odot f) = c\odot I_{j}(\mu)( f)$ for any  $c\in\mathbb{R}$, $f,g\in  C(X, \mathbb{R})$ and $j\in J$.
It follows that
\[ \bigoplus_{j \in J} I_{j}(\mu)(f\oplus g) = \bigoplus_{j \in J} [I_{j}(\mu)(f)\oplus I_{j}(\mu)( g)] = \left[\bigoplus_{j \in J} I_{j}(\mu)(f)\right]\oplus\left[ \bigoplus_{j \in J}I_{j}(\mu)( g)\right]\]
and
\[\bigoplus_{j \in J} I_{j}(\mu)(c\odot f) = \bigoplus_{j \in J} [c\odot I_{j}(\mu)(f)] = c\odot \bigoplus_{j \in J} I_{j}(\mu)(f).\]

Now we prove that $M_{\phi,q}(\mu)(0) = 0$. Let $\lambda\in U(X, \mathbb{R}_{\max}) $ be the density of $\mu$. As $\mu\in I(X)$ there exists $x_0\in X$ such that $\lambda(x_0) = \oplus_x \lambda(x) = 0$. Then we have
\[M_{\phi,q}(\mu)(0) =  \bigoplus_{j \in J} I_{j}(\mu)(0) =  \bigoplus_{j \in J} \mu( q_{j} ) = \bigoplus_{j \in J} \bigoplus_{x\in X} [\lambda(x) \odot  q_{j}(x)].\]
As $\lambda \leq 0$ and  $q\leq 0$ (due equation \eqref{eq: q normalized}) we get $M_{\phi,q}(\mu)(0)\leq 0$. On the other hand,
\[M_{\phi,q}(\mu)(0) = \bigoplus_{j \in J} \bigoplus_{x\in X} [\lambda(x) \odot  q_{j}(x)] \geq \bigoplus_{j \in J} [\lambda(x_0) \odot  q_{j}(x_0)] = \bigoplus_{j \in J}  q_{j}(x_0)\stackrel{\eqref{eq: q normalized}}{=}0.\]

\end{proof}

\begin{definition}\label{def:bounded functions usual}
We denote by $\mathcal{B}(X, \mathbb{R}_{\max})$  the set of functions from $X$ to $\mathbb{R}_{\max}$ bounded from above.
\end{definition}

From now on we adopt the convention $\bigoplus_{y\in A} f(y) = -\infty$ if $A=\varnothing$.

\begin{definition}\label{def:transfer operator}
   To each mpIFS  we assign a max-plus transfer operator  $L_{\phi,q}: \mathcal{B}(X, \mathbb{R}_{\max}) \to \mathcal{B}(X, \mathbb{R}_{\max})$ defined by
   \begin{equation}\label{eq:transfer operator}
      L_{\phi,q}(\lambda)(x):=\bigoplus_{(j,y)\in \phi^{-1}(x)} q_{j}(y) \odot  \lambda(y)
   \end{equation}
  (if $\{(j,y) | \phi_{j}(y)=x\}=\varnothing$ then $L_{\phi,q}(\lambda)(x)=-\infty$).
\end{definition}

We have $ U(X, \mathbb{R}_{\max}) \subset \mathcal{B}(X, \mathbb{R}_{\max})$ and next proposition shows that we could alternatively  consider $L_{\phi,q}: U(X, \mathbb{R}_{\max}) \to U(X, \mathbb{R}_{\max})$ in Definition \ref{def:transfer operator}.

\begin{proposition} If $\lambda \in  U(X, \mathbb{R}_{\max})$ then $L_{\phi,q}(\lambda)\in  U(X, \mathbb{R}_{\max})$.
	\end{proposition}
\begin{proof}
	Initially suppose that $x_0$ is such that $\{(j,y) | \phi_{j}(y)=x_0\}=\varnothing$. Then $L_{\phi,q}(\lambda)(x_0) = -\infty$ and we need to prove that $L_{\phi,q}(\lambda)(x)=-\infty$ for any $x$ in a open set containing $x_0$ (see Definition \ref{def:usc}). This is a direct consequence of the continuity of $\phi:J\times X \to X$ where $J$ and $X$ are compact spaces (its image is closed).
	 	
	Now we suppose that  $x_0$ is such that $\{(j,y) | \phi_{j}(y)=x_0\}\neq\varnothing$. Then     $L_{\phi,q}(\lambda)(x_0):=\bigoplus_{(j,y) \in \phi^{-1}(x_0)} q_{j}(y) \odot  \lambda(y)$. Let $x_n$ be a sequence converging to $x_0$. We want to prove that $\limsup_{n \to \infty} L_{\phi,q}(\lambda)(x_n)\leq L_{\phi,q}(\lambda)(x_0)$ (see Definition \ref{def:usc}). We can suppose that for any $n$ the set $\{(j,y)|\phi_j(y)=x_n\}$ is non empty. Given $\varepsilon>0$ let $(j_n,y_n)$ be such that $\phi_{j_n}(y_n)=x_n$ and $L_{\phi,q}(\lambda)(x_n) \leq q_{j_n}(y_n) \odot  \lambda(y_n)+\varepsilon$. We can suppose there exists $j_0,y_0$ such that $j_n\to j_0$ and $y_n\to y_0$. As $\phi:J\times X \to X$ is continuous we have $x_0 = \phi_{j_0}(y_0)$. Then, as $\lambda$ is u.s.c., we get
	\[\limsup_{n \to \infty} L_{\phi,q}(\lambda)(x_n) \leq \limsup_{n \to \infty} [q_{j_n}(y_n) \odot  \lambda(y_n)]+\varepsilon  \leq q_{j_0}(y_0) \odot \lambda(y_0)+\varepsilon \leq L_{\phi,q}(\lambda)(x_0)+\varepsilon.\]
\end{proof}

Next theorem exhibits the relation between definitions \ref{def:Markov operator}, \ref{def:transfer operator} and \ref{def:mp ruelle operator}.

\begin{theorem}\label{thm: equivalences} Given a function $\lambda \in  U(X, \mathbb{R}_{\max})$ satisfying $\oplus_x \lambda(x) = 0$ and the associated idempotent probability   $\mu=\bigoplus_{x\in X}\lambda(x)\odot\delta_x\in I(X)$ we have that $M_{\phi,q}(\mu)= \bigoplus_{x \in X} L_{\phi,q} (\lambda)(x) \odot \delta_x$, that is, $M_{\phi,q}(\mu)$ has density $L_{\phi,q} (\lambda)$ where $\lambda$ is the density of $\mu$. Furthermore
	$$M_{\phi,q}(\mu)(f) = \mu(\mathcal{L}_{\phi,q}(f)),$$
	for any $f \in C(X, \mathbb{R})$, that is, $M_{\phi,q}$ is the max-plus dual of $\mathcal{L}_{\phi,q}$.
An idempotent probability $\mu$ is invariant if and only if its density $\lambda$ is invariant for $L_{\phi,q}$ that is, it satisfies
\begin{equation}\label{eq:lambda invariant}
		\lambda(x) = \bigoplus_{(j,y)\in\phi^{-1}(x)} q_{j}(y) \odot  \lambda(y).
		\end{equation}
\end{theorem}
\begin{proof}
We have
\[    M_{\phi,q}(\mu) (f)=  \bigoplus_{j \in J} \bigoplus_{y\in X}\left(\lambda(y)\odot q_{j}(y)\odot f(\phi_{j}(y))\right)=\bigoplus_{x\in X}\bigoplus_{(j,y)\in\phi^{-1}(x)} \left(\lambda(y)\odot q_{j}(y)\odot f(x)\right)\]
\[ =\bigoplus_{x\in X}\left(\left(\bigoplus_{(j,y)\in \phi^{-1}(x)} \left(\lambda(y)\odot q_{j}(y)\right)\right)\odot f(x)\right)=\bigoplus_{x \in X} L_{\phi,q} (\lambda)(x) \odot f(x).\]
Then $M_{\phi,q}(\mu)= \bigoplus_{x \in X} L_{\phi,q} (\lambda)(x) \odot \delta_x$ has density $L_{\phi,q} (\lambda) \in U(X, \mathbb{R}_{\max})$ and applying Theorem \ref{teo : densidade Maslov} we conclude that $M_{\phi,q}(\mu)=\mu$ if and only if $L_{\phi,q} (\lambda)=\lambda$.
	
Now we will prove that $M_{\phi,q}(\mu)(f) = \mu(\mathcal{L}_{\phi,q}(f)),$
for any $f \in C(X, \mathbb{R})  $.  Denoting by $\lambda$ the density of $\mu$ we have	
$$M_{\phi,q}(\mu)(f)= \bigoplus_{j \in J} I_{j}(\mu) (f)=
\bigoplus_{j \in J} \mu(q_{j}\odot (f\circ\phi_{j}))= \bigoplus_{j \in J} \bigoplus_{x\in X}\lambda(x)\odot q_{j}(x)\odot f(\phi_{j}(x))$$
$$ =\bigoplus_{x\in X}\bigoplus_{j \in J}\lambda(x)\odot q_{j}(x)\odot f(\phi_{j}(x))= \bigoplus_{x\in X}\lambda(x) \odot \left(\bigoplus_{j \in J} q_{j}(x)\odot f(\phi_{j}(x))\right) = \mu(\mathcal{L}_{\phi,q}(f)).$$
\end{proof}

Next result will be useful in the proof of Theorem \ref{teo: main result}.

\begin{corollary} \label{cor: inv Lq is inv M} Given $\lambda \in \mathcal{B}(X, \mathbb{R}_{\max})$ such that $\bigoplus_{x\in X}\lambda(x) = 0$, let $\mu:C(X,\mathbb{R})\to\mathbb{R}$ be given by
	\[\mu(f) := \bigoplus_{x\in X} \lambda(x)\odot f(x).\]
	Then $\mu \in I(X)$. Furthermore, given a mpIFS $\mathcal{S}=(X, \phi, q)$, if $L_{\phi,q}(\lambda)=\lambda$ then $M_{\phi,q}(\mu) = \mu$.
\end{corollary}

\begin{proof}
	It is immediate to check that $\mu \in I(X)$. Furthermore the computations at the beginning of the proof of Theorem \ref{thm: equivalences} can be applied for $\lambda \in \mathcal{B}(X, \mathbb{R}_{\max})$ in order to conclude that
	\[    M_{\phi,q}(\mu) (f)= \bigoplus_{x \in X} L_{\phi,q} (\lambda)(x) \odot f(x).\]
	Supposing then  $L_{\phi,q}(\lambda)=\lambda$ we get
		\[    M_{\phi,q}(\mu) (f)= \bigoplus_{x \in X} L_{\phi,q} (\lambda)(x) \odot f(x) =\bigoplus_{x\in X}\lambda(x)\odot f(x) = \mu(f) .\]
\end{proof}

\section{Invariant probabilities for place dependent mpIFS}

\subsection{Non uniqueness of the invariant idempotent probability}\label{Non uniqueness of the invariant idempotent measure for mpIS with variable weights}
In this subsection we provide an example of a mpIFS with place dependent weights having infinitely many invariant idempotent probabilities, in contrast with the constant case from \cite{MZ}, where the invariant idempotent probability is unique. The example we build is very regular, raising the question if there are some additional constraints to hold uniqueness (other than to be constant), and showing that some alternative tool should be put in place to describe these invariant measure in the general setting. We will do that in the Subsection~\ref{sec:Representation of max-plus  invariant densities}.

We consider the mpIFS $\mathcal{S}=(X, \phi, q)$ and the max-plus operator $M_{\phi,q}:I(X)\to I(X)$ given by (see Theorem \ref{thm: equivalences} )
\[M_{\phi,q}(\mu)(f)=\mu(\mathcal{L}_{\phi,q}(f))=\mu(\oplus_{j\in J} [q_j(x) \odot f(\phi_j(x))]).\]

In such operator the weight $q_j(x) \leq 0$ depends of $x$. We will present an example of mpIFS in a such way that $M_{\phi,q}$ has infinitely many invariant idempotent probabilities.

Consider the space $X=\{1,2\}^{\mathbb{N}}$ with the metric $d$ defined by
\[d((x_1,x_2,x_3,...),(y_1,y_2,y_3,...)) = \left(1/2\right)^{min\{i\,|\,x_i\neq y_i\}},\,x\neq y.\]
Let $J=\{1,2\}$ be a set of indices, with the metric satisfying $d(1,2) = 1$, and let us consider the uniformly contractive IFS $\phi_j(x_1,x_2,x_3,...) = (j,x_1,x_2,...)$, which is defined as the inverse branches of the shift map.

Let \[q_j(x) = q_j(x_1,x_2,x_3,...) = \left\{\begin{array}{ll} 0& \text{if}\,\,j=x_1 \\ -1 & \text{if}\,\,j\neq x_1\end{array}\right..\]
Then $q_1(1,x_2,x_3,...) = q_2(2,x_2,x_3,...) = 0$ while $q_1(2,x_2,x_3,...) = q_2(1,x_2,x_3,...) = -1$.

Given $\alpha \in [0,1)$, let $\lambda_\alpha:X\to\mathbb{R}_{\max}$ be defined by:\newline
$\lambda_\alpha(x) = -\infty$ if $x$ has infinitely many 2's and 1's.\newline
$\lambda_\alpha(x) = -n$ if $x$ has $n$ changes of symbols and finitely many 2's.\newline
$\lambda_\alpha(x) = -n-\alpha$ if  $x$ has $n$ changes of symbols and finitely many 1's.\newline
$\lambda_\alpha(1^{\infty})=0$ and $\lambda_\alpha(2^{\infty}) = -\alpha$, where $2^{\infty} := (2,2,2,....)$ and $1^{\infty} := (1,1,1,....)$.

For example,
$$\lambda_\alpha(2,2,2,1^{\infty}) =-1, \; \lambda_\alpha(2,1,1,2,1^{\infty}) =-3 \text{ and } \lambda_\alpha(2,2,2,1,2^\infty)=-2-\alpha.$$

If $x_n\to x$ then for any given $k>0$ we get that the points $x_n$ and $x$ of $\{1,2\}^{\mathbb{N}}$ have the same initial $k$ symbols of $\{1,2\}$, if $n$ is large enough. From this remark and analyzing the cases in definition of $\lambda_\alpha$ it is easy to conclude that $\lambda_\alpha \in U(X, \mathbb{R}_{\max})$. Let us define $\mu_\alpha \in I(X)$ by
\[\mu_\alpha(f) := \oplus_{x \in X} (\lambda_\alpha(x)\odot f(x)).\]

\begin{proposition} For any $\alpha \in [0,1)$ the idempotent probability $\mu_\alpha$, as above defined, is invariant  for the operator $M_{\phi,q}$.
\end{proposition}
\begin{proof}
	The key point in the proof is to observe that $\lambda_\alpha$ satisfies
	\begin{equation}\label{eq1.0}
		\lambda_\alpha(x)=q_j(y)\odot \lambda_\alpha(y),\,\,\,\,\text{if}\,\,\,\phi_j(y)=x,
	\end{equation}
	which is a direct consequence of the construction of $\lambda_\alpha$, meaning that $L_{\phi,q}(\lambda_\alpha)=\lambda_\alpha$. Then the proof follows by applying Theorem \ref{thm: equivalences}.
\end{proof}	
	
\subsection{Representation of max-plus  invariant densities}\label{sec:Representation of max-plus  invariant densities}

In this subsection we characterize the invariant idempotent probabilities (see Theorem \ref{teo: main result}). In Theorem \ref{thm: equivalences} it was proven that an idempotent probability is invariant if and
only if its density $\lambda$ is invariant for $L_{\phi,q}$, that is, it satisfies
\eqref{eq:lambda invariant}.  This equation will be used to characterize the density on its support by adapting to our setting the notions of Ma\~{n}\'{e} potential and Aubry set.

We consider a mpIFS $\mathcal{S}=(X, (\phi_j)_{j\in J}, (q_j)_{j\in J})$ and assume that $\mu=\bigoplus_{x\in X}\lambda(x)\odot\delta_x\in I(X)$ is invariant, that is,  $M_{\phi,q}(\mu)=\mu$, or equivalently, its density $\lambda$ satisfies $L_{\phi,q} (\lambda) = \lambda$, where
$$L_{\phi,q} (\lambda)(x)=\bigoplus_{(j,y)\in\phi^{-1}(x)} q_j(y) \odot  \lambda(y).$$

 Initially  we suppose that $\lambda$ satisfies the equation
$$\lambda(x)=\bigoplus_{(j,y)\in\phi^{-1}(x)} q_j(y) \odot  \lambda(y).$$
We fix $x \in X$ such that $\lambda(x)>-\infty$ and suppose there is $(j_1,y_1)\in \phi^{-1}(x)$ such that $\lambda(x)=q_{j_1}(y_{1}) \odot  \lambda(y_{1})$.
 Analogously,  we suppose  there exists $(j_{2},y_2)\in\phi^{-1}(y_1)$ such that
$$\lambda(y_{1})= q_{j_{2}}(y_{2}) \odot  \lambda(y_{2}).$$
Substituting that in the previous inequality and using $y_{1}=\phi_{j_{2}}(y_{2})$ we obtain
$$\lambda(x)= q_{j_1}(\phi_{j_{2}}(y_{2})) \odot  q_{j_{2}}(y_{2}) \odot  \lambda(y_{2}).$$
A repetition of this argument, if possible, produces
$$\lambda(x)= q_{j_1}(\phi_{j_{2}}(\phi_{j_{3}}(y_{3}))) \odot  q_{j_{2}}(\phi_{j_{3}}(y_{3})) \odot q_{j_{3}}(y_{3}) \odot \lambda(y_{3})$$
and more generally, for each $n$,
$$\lambda(x)=q_{j_1}(\phi_{j_2}\circ \cdots \circ\phi_{j_n} (y_n))\odot q_{j_2}(\phi_{j_3}\circ \cdots \circ\phi_{j_n} (y_n))\odot \dots \odot q_{j_n}(y_n)\odot \lambda(y_{n}).$$

Inspired by the above computations (see equation~\eqref{eq: represent density} ahead), we consider the following approach.

\begin{notation}
	Given $x \in X$, $n\in\mathbb{N}$ and a finite sequence $\omega = (j_1,j_2,...,j_n)\in J^{n}$ we denote
	$$\phi_{\omega}(x)=\phi_{(j_1,...,j_n)}(x) := \phi_{j_1}\circ\dots\circ \phi_{j_n}(x)$$
	and
	$${\rm Sum}(\omega,x) :=  q_{j_1}(\phi_{(j_2,...j_n)} (x))\odot q_{j_2}(\phi_{(j_3,...,j_n)} (x))\odot\dots\odot q_{j_n}(x).$$	
\end{notation}

Given $x,y \in X$ and $\varepsilon>0$ we define
\begin{equation}\label{def: S_epsilon} S_\varepsilon(x,y) = \bigoplus_{n\in\mathbb{N}}\left[\bigoplus_{\omega \in J^n \,|\,d(x,\phi_{\omega}(y))<\varepsilon}{\rm Sum}(\omega,y)\right]
\end{equation}
which can be $-\infty$ if the set $\{\omega \in J^n\,|\,d(x,\phi_{\omega}(y))<\varepsilon\}$ is empty for any $n$. Clearly $0<\epsilon < \epsilon' \Rightarrow S_\varepsilon\leq S_{\varepsilon '}$, so we can define the \textit{Ma\~{n}\'{e} potential} $S:X\times X \to [-\infty,0]$, by
\[S(x,y) =\lim_{\varepsilon\to 0} S_\varepsilon(x,y).\]
Let us consider also the  \textit{Aubry set}
\[\Omega = \{x\in X | S(x,x) = 0\}.\]

Historically, the  Ma\~{n}\'{e} potential and the Aubry set were introduced by R. Ma\~{n}\'{e}  in the 90's to study the Aubry-Mather theory (see \cite{Man97}), that is, to characterize the minimizing invariant measures for a Lagrangrian flow. Later, this idea was extended by many authors for several discrete settings such as ergodic optimization (see, for instance, \cite[Definition 24]{CLT01}, where it is denoted as ``\emph{action potential}"), where the purpose is to find invariant measures minimizing a given potential function (see also  \cite{BLL}, \cite{GaribLop08} and \cite{RenArtJai18}).

\begin{lemma}\label{lemma: Sum lipschitz} Given a mpIFS $\mathcal{S}=(X, \phi, q)$, for any $y_1,y_2 \in X$ and $\omega \in J^n$ we have
	\[|{\rm Sum}(\omega,y_1)-{\rm Sum}(\omega,y_2)|< \frac{C}{1-\gamma}d(y_1,y_2),\]
where $\gamma$ and $C$ satisfy \eqref{eq: gamma contraction} and \eqref{eq: q lipschitz}.
\end{lemma} 	
\begin{proof}
From  \eqref{eq: gamma contraction} and \eqref{eq: q lipschitz} we get
\[|{\rm Sum}(\omega,y_1)-{\rm Sum}(\omega,y_2)| \leq \sum_{i=1}^{n} C\cdot\gamma^i \cdot d(y_1,y_2)\leq \frac{C}{1-\gamma}d(y_1,y_2).\]

\end{proof}

\begin{proposition} Given a mpIFS $\mathcal{S}=(X, \phi, q)$, the Aubry set $\Omega$ is non empty.
\end{proposition}
\begin{proof}
Let $x\in X$ be any point. Let $(j_n)$ be a sequence of points of $J$ satisfying $q_{j_1}(x) = 0 $ and $q_{j_{n+1}}(\phi_{j_n}\circ...\circ\phi_{j_1}(x)) = 0$ for all $n\geq 1$ (such sequence exists from equation \eqref{eq: q normalized}, where $J$ is compact and $q$ is continuous).  Let $\omega_{k,l}:=(j_{k},j_{k-1},...,j_{l})$ for $k\geq l$ and $x_n :=\phi_{\omega_{n,1}}(x)= \phi_{j_n}\circ...\circ\phi_{j_1}(x)$. Then, it follows that $q_{j_{n+1}}(x_n) = 0$ and
	\begin{equation}\label{eq1_aubry}
		 {\rm Sum}(\omega_{k,n+1},x_n)=0,\,\,\,\forall k>n.
	  \end{equation}
	
As $X$ is compact there exists a subsequence $(n_i)$ and a point $\tilde{x}$ such that $x_{n_i}\to \tilde{x}$. We will prove that $\tilde{x} \in \Omega$.

Given $\varepsilon>0$ there exist $n,m\in\mathbb{N}$ such that $n<m$ and $x_n, x_m \in B(\tilde{x},\varepsilon/2)$. Observe that $x_m = \phi_{\omega_{m,n+1}}(x_n)$ and as $\phi$ satisfies \eqref{eq: gamma contraction}   we get 	
\[d(\tilde{x},\phi_{\omega_{m,n+1}}(\tilde{x})) \leq d(\tilde{x},x_m)+ d(x_m, \phi_{\omega_{m,n+1}}(\tilde{x})) =d(\tilde{x},x_m)+ d(\phi_{\omega_{m,n+1}}(x_n), \phi_{\omega_{m,n+1}}(\tilde{x}))<\varepsilon.\]
It follows that
\[S_{\varepsilon}(\tilde{x},\tilde{x})\geq {\rm Sum}(\omega_{m,n+1},\tilde{x}).\]
Therefore, applying Lemma \ref{lemma: Sum lipschitz} and equation \eqref{eq1_aubry} we have
\[S_{\varepsilon}(\tilde{x},\tilde{x})\geq - \frac{C}{1-\gamma}\cdot \frac{\varepsilon}{2}\]
and so $S(\tilde{x},\tilde{x}) =\lim_{\varepsilon\to 0} S_{\varepsilon}(\tilde{x},\tilde{x}) \geq 0.$ As we always have  $S=\lim_{\varepsilon\to 0} S_{\varepsilon}\leq 0$, by \eqref{eq: q normalized}, we complete the proof.
\end{proof}

The next proposition is somehow analogous to the one founded in \cite[Proposition 23, (ii)]{CLT01}.

\begin{proposition}\label{prop:triangular} For any $x,y,z\in X$  we have
	\[S(x,y)\odot S(y,z) \leq S(x,z).\]
\end{proposition}

\begin{proof}
We can suppose $S(x,y)>-\infty$ and $S(y,z)>-\infty$. As $0<\epsilon_1<\epsilon_2 \Rightarrow S_{\epsilon_1}\leq S_{\epsilon_2}$ and $S=\lim_{\epsilon \to 0}S_{\epsilon}$ we get $S_{\epsilon}(x,y)>-\infty$ and $S_{\epsilon}(y,z)>-\infty$ for any $\epsilon>0$. Given any $\varepsilon>0$, as $S_{\epsilon/2}$ is obtained by a supremum (see \eqref{def: S_epsilon}), there exist finite sequences $\omega=(i_1,...,i_n)$ and $\eta=(j_1,...,j_m)$ such that
	$d(x, \phi_{\omega}(y)) < \varepsilon/2$, $d(y, \phi_{\eta} (z))<\varepsilon/2$
	and
	\[S_{\varepsilon/2}(x,y) + S_{\varepsilon/2}(y,z)-\varepsilon < {\rm Sum}(\omega,y)+{\rm Sum}(\eta,z).\]
	We denote by $\omega\circ\eta:=(i_1,...,i_n,j_1,...,j_m)$. Consider the points $\tilde{y} =\phi_\eta(z)$ and $\tilde{x} = \phi_\omega(\tilde{y}) = \phi_\omega\circ\phi_\eta(z)=\phi_{\omega\circ\eta}(z)$. As $\phi$ is a contraction we get
	\[d(x,\tilde{x})\leq d(x,\phi_\omega(y))+ d(\phi_\omega(y),\phi_\omega(\tilde{y}))<\frac{\varepsilon}{2}+ \gamma^{n}\frac{\varepsilon}{2} <  \varepsilon.\]
It follows that
	\[S_\varepsilon(x,z) \geq {\rm Sum}(\omega\circ\eta, z) = {\rm Sum}(\omega,\phi_\eta(z))\odot  {\rm Sum}(\eta,z) = {\rm Sum}(\omega,\tilde{y})\odot  {\rm Sum}(\eta,z).\]
	By applying Lemma \ref{lemma: Sum lipschitz} we get
\[S_\varepsilon(x,z) \geq -\frac{C}{1-\gamma}d(y,\tilde{y})\odot  {\rm Sum}(\omega,y)\odot  {\rm Sum}(\eta,z).\]
  Therefore
	\[S_\varepsilon(x,z) > -\frac{C}{1-\gamma}\frac{\varepsilon}{2} \odot  S_{\varepsilon/2}(x,y) \odot  S_{\varepsilon/2}(y,z)-\varepsilon.\]
Finally, as $\varepsilon \to 0$ we get
	\[S(x,z) \geq S(x,y) \odot  S(y,z).\]
\end{proof}

Using the Ma\~{n}\'{e} potential $S(\cdot,\cdot)$ and the Aubry set $\Omega$ we are able to build candidates to be densities of invariant idempotent probabilities. In some sense, those are determined by its value at some point of the Aubry set plus the value of the Ma\~{n}\'{e} potential between these two points.

Let $x_0$ be a point of $\Omega$. Then consider any function $\lambda_1:\Omega \to [-\infty, 0]$ satisfying $\lambda_1(x_0) = 0$.
Let $\overline{\lambda}:X \to [-\infty,0]$ be the function defined by
\[\overline{\lambda}(x) := \bigoplus_{y\in \Omega}[ S(x,y) \odot \lambda_{1}(y)].\]
Observe that $\overline{\lambda}\leq 0$ because $S\leq 0$ and $\lambda_1\leq 0$. We also have $\overline{\lambda}(x_0)=0$ because
\[\overline{\lambda}(x_0) \geq S(x_0,x_0)\odot\lambda_1(x_0) = \lambda_1(x_0)=0.\]

We claim that the function $\overline{\lambda}$ satisfies
\[\overline{\lambda}(x) = \bigoplus_{z\in \Omega}( S(x,z) \odot \overline{\lambda}(z)).\]
Indeed,  for any $y\in \Omega$ we have
\[S(x,y) \stackrel{Prop. \,\ref{prop:triangular}}{\geq} \sup_{z \in \Omega} S(x,z)\odot S(z,y) \stackrel{y\in \Omega}{\geq}  S(x,y)\odot S(y,y) \stackrel{y\in \Omega}{=} S(x,y).\]
Consequently  $S(x,y) = \sup_{z \in \Omega} S(x,z)\odot S(z,y), \,\,\forall y\in \Omega$. Then\footnote{For any sets $A$ and $B$ and any function $f:A\times B \to \mathbb{R}_{\max}$ bounded from above, by definition of supremum, we have\,\, $\displaystyle{\bigoplus_{a\in A}\left[\bigoplus_{b\in B}f(a,b)\right] = \bigoplus_{(a,b)\in A\times B}f(a,b) = \bigoplus_{b\in B}\left[\bigoplus_{a\in A}f(a,b)\right].}$}
\[\overline{\lambda}(x) = \bigoplus_{y\in \Omega}( S(x,y) \odot \lambda_1(y)) = \bigoplus_{y\in \Omega}\bigoplus_{z\in\Omega}( S(x,z) \odot S(z,y)\odot \lambda_1(y)) \]\[=\bigoplus_{z\in \Omega}( S(x,z) \odot \bigoplus_{y\in \Omega} ( S(z,y)\odot  \lambda_1(y))) = \bigoplus_{z\in \Omega}( S(x,z) \odot \overline{\lambda}(z)).\]
This proves the claim.

\bigskip

Now we are able to prove the main result of the present work. We will show that for any function $\lambda_1$ as above and the associated function $\overline{\lambda}$ we get an idempotent probability  which is invariant for $M_{\phi,q}$ and reciprocally.  In the theorem below we are not claiming that $\overline{\lambda}$ is u.s.c. which is not a problem, since from Corollary \ref{cor: inv Lq is inv M} we only need $\overline{\lambda}$ to be bounded from above to define an idempotent probability. Corollary \ref{cor: inv Lq is inv M} plays an important hole in the proof of the first part of next theorem.

\begin{theorem}\label{teo: main result} Suppose that a function $\overline{\lambda}:X\to[-\infty,0]$ satisfies $\bigoplus_{x\in X}\overline{\lambda}(x) = 0$ and
	\[\overline{\lambda}(x) = \bigoplus_{z\in \Omega}( S(x,z) \odot \overline{\lambda}(z)) . \]
Then the functional $\overline{\mu}\in C^*(X,\mathbb{R})$, defined by $\overline{\mu}(f) :=\bigoplus_{x\in X} (\overline{\lambda}(x)\odot f(x))$, is an idempotent probability  which is invariant for the operator $M_{\phi,q}$. Reciprocally, if $\mu \in I(X)$ is invariant for $M_{\phi,q}$ and $\lambda \in U(X,\mathbb{R}_{\max})$ is its (unique u.s.c.) density, then $\lambda$ satisfies the equation
	\begin{equation}\label{eq: represent density}
	  \lambda(x) = \bigoplus_{z\in \Omega} [S(x,z) \odot \lambda(z)].
	\end{equation}
\end{theorem}
\begin{proof}  We will prove that $\overline{\lambda}$ is invariant for the operator $L_{\phi,q}$ and then apply Corollary \ref{cor: inv Lq is inv M}.  In this way, we need to show that $\overline{\lambda}$ satisfies, for any $x\in X$,
$$ \bigoplus_{(j,y)\in\phi^{-1}(x)} q_j(y) \odot   \overline{\lambda}(y)=  \overline{\lambda}(x). $$
Equivalently, for any $x\in X$,
$$ \bigoplus_{(j,y)\in\phi^{-1}(x)} \left(q_j(y) \odot  \bigoplus_{z\in \Omega} [S(y,z) \odot   \overline{\lambda}(z)]\right)=  \bigoplus_{w\in \Omega} [S(x,w) \odot   \overline{\lambda}(w)]. $$
We will prove that for any $z\in \Omega$ and $x\in X$ we have
$$ \bigoplus_{(j,y)\in\phi^{-1}(x)} q_j(y) \odot   S(y,z) =   S(x,z).$$

Given $x\in X$, if there is not $(j,y)$ such that $\phi_j(y)=x$ then, using that $\phi:J\times X\to X$ has a closed image (it is continuous with a compact domain), there exists and $\varepsilon>0$ such that
$$d(x,\tilde{x})<\varepsilon \Rightarrow [\nexists \,(j,y)\,\,s.t. \,\,\phi_j(y)=\tilde{x}].$$
In this case $S_{\varepsilon}(x,z)=-\infty$ for any $z\in X$ and the equation $$ \bigoplus_{(j,y)\in\phi^{-1}(x)} q_j(y) \odot   S(y,z) =   S(x,z) $$ is satisfied as $-\infty=-\infty$.

Given $x \in X, \, z\in \Omega$ and $y,j$ such that $\phi_{j}(y)=x$ we will prove now that $S(x,z) \geq q_j(y) \odot   S(y,z)$.
Supposing $S(y,z)\neq -\infty$, from the definition of $S(\cdot,\cdot)$ there exists a sequence $y_n\to y$ in the form $y_n = \phi_{\omega_n}(z)$
such that $${\rm Sum}(\omega_n,z)\to S(y,z).$$
Let $x_n =\phi_j(y_n)$. Then $x_n \to x$ and $x_n = \phi_j\circ\phi_{\omega_n}(z)$. By definition of $S(x,z)$ we have
\[S(x,z) \geq  \limsup_{n\to\infty} [q_{j}(\phi_{\omega_n}(z))\odot  {\rm Sum}(\omega_n,z)] = q_j(y) \odot   S(y,z).\]

Now, given $x \in X$ such that $\{(j,y)\,|\,\phi_j(y)=x\}\neq \varnothing$ and $z\in \Omega$ we will show that there exist $y,j$ satisfying $\phi_{j}(y)=x$ and $S(x,z) \leq q_j(y) \odot   S(y,z)$.
From the definition of $S(\cdot,\cdot)$ there exists a sequence $\omega_n = (j_1^n,...,j_{k(n)}^n)$ such that $\phi_{\omega_n}(z) \to x$ and ${\rm Sum}(\omega_n,z) \to S(x,z)$. We define $\eta_n :=(j_2^n,...,j_{k(n)}^n)$, $x_n := \phi_{\omega_n}(z)$ and $y_n := \phi_{\eta_n}(z)$ (if $k(n)=1$ for some $n$ then $\eta_n$ doesn't exist and by convention we consider $y_n = z$ and denote ${\rm Sum}(\eta_n,z)=0$).

By taking a subsequence we can suppose there exists $j\in J$ such that $j_1^n \to j$, as $n\to\infty$.  By taking a subsequence again we can suppose there exists $y\in X$ such that $y_n \to y$ as $n\to\infty$. As $\phi$ is continuous in $J\times X$ we have $\phi_j(y) = x$. Furthermore,
\[q_{j}(y) \odot   S(y,z) \geq q_{j}(y) \odot   \limsup_{n\to\infty} \,{\rm Sum}(\eta_n,z)= \limsup_{n\to\infty} [q_{j_1^n}(y_n) \odot  {\rm Sum}(\eta_n,z)]\]
\[=\limsup_{n\to\infty}\, {\rm Sum}(\omega_n,z) = S(x,z).\]

These cases conclude the proof that $L_{\phi,q}(\overline{\lambda})=\overline{\lambda}$. By applying Corollary \ref{cor: inv Lq is inv M} we conclude that $\overline{\mu}(f) :=\bigoplus_{x\in X} (\overline{\lambda}(x)\odot f(x))$ is an idempotent probability  satisfying $M_{\phi,q}(\overline{\mu})=\overline{\mu}$.

Reciprocally, suppose now that $\mu \in I(X)$ satisfies $M_{\phi,q}(\mu)=\mu$. Let $\lambda\in U(X,\mathbb{R}_{\max}) $ be its unique density according to Theorem \ref{teo : densidade Maslov}. By applying Theorem \ref{thm: equivalences} we get that
\begin{equation}\label{eq1}
	\bigoplus_{(j,y)\in \phi^{-1}(x)} q_j(y) \odot   \lambda(y)=  \lambda(x).
\end{equation}
It is necessary to prove that
	\[\lambda(x) = \bigoplus_{z\in \Omega} [S(x,z) \odot   \lambda(z)].\]
We start by proving that for any $z \in \Omega$ we have	
\[\lambda(x) \geq  S(x,z) \odot   \lambda(z).\]
 Indeed, if there is not a sequence of points in the form $x_n = \phi_{\omega_n}(z) =\phi_{j_1^n}\circ...\circ \phi_{j_{k(n)}^n}( z)$ such that $x_n \to x$ then $S(x,z)=-\infty$ and the statement is satisfied. Suppose now that there exists a such sequence. Then by applying recursively equation \eqref{eq1} we have
\[\lambda(x_n) \geq  q_{j_1^n}(\phi_{j_2^n}\circ...\circ \phi_{j_{k(n)}^n} (z)) \odot  q_{j_2^n}( \phi_{j_3^n}\circ...\circ \phi_{j_{k(n)}^n} (z))\odot  q_{j_{k(n)}^n}(z) \odot   \lambda(z) = {\rm Sum}(\omega_n,z)\odot  \lambda(z).\]
As $x_n\to x$ and $\lambda$ is u.s.c. we get
\[\lambda(x) \geq \limsup_{n\to\infty} \lambda(x_n) \geq 	\limsup_{n\to\infty}\, {\rm Sum}(\omega_n,z) \odot   \lambda(z).\]
Taking the supremum over any such sequence $\omega_n$ we get
\[\lambda(x) \geq S(x,z) \odot   \lambda(z).\]

Now we will to prove the reverse inequality. Supposing that $\lambda(x_0)>-\infty$, we will show that, for each $\varepsilon>0$, there exists a point $z\in \Omega$ such that
\[\lambda(x_0) \leq  S(x,z) \odot   \lambda(z)\odot  \varepsilon.\]

Given $\varepsilon>0$, by applying equation \eqref{eq1} recursively, we obtain sequences $x_1,x_2,...$ and $j_0,j_1,j_2,...$ satisfying $\phi_{j_{n-1}}(x_n) =x_{n-1}$ and
\[q_{j_{n-1}}(x_n) \odot   \lambda(x_n)\geq  \lambda(x_{n-1}) -\varepsilon\cdot 2^{-n}.\]
It follows that
\[\lambda(x_0) \leq q_{j_0}(x_1) \odot   \lambda(x_1) \odot  \varepsilon \cdot 2^{-1} \leq q_{j_0}(x_1) \odot   q_{j_1}(x_2)\odot   \lambda(x_2) \odot  \varepsilon \cdot (2^{-1}\odot 2^{-2})\]
\[\leq q_{j_0}(x_1) \odot   q_{j_1}(x_2)\odot   q_{j_2}(x_3)\odot   \lambda(x_3) \odot  \varepsilon \cdot (2^{-1}\odot  2^{-2}\odot  2^{-3})\leq...,\]
and so
\begin{equation}\label{eq: lambda 0 first}
	\lambda(x_0) \leq q_{j_0}(x_1) \odot   q_{j_1}(x_2)\odot  ...\odot   q_{j_{n-1}}(x_{n})\odot   \lambda(x_n) \odot  \varepsilon,\,\,\forall\,n\geq 1.
	\end{equation}

As $\lambda \leq 0$ and $\lambda(x_0)>-\infty$ we conclude also that
\begin{equation}\label{eq: lambda 0}
	-\infty < \lambda(x_0) \leq  \left(\sum_{n\geq 1} q_{j_{n-1}}(x_n) \right) \odot   \varepsilon  .
	\end{equation}
Given numbers $n,m \in \mathbb{Z}$ such that $0\leq n<m$ we define $\omega_{n,m} := (j_n,...,j_{m-1})$ and observe that $\phi_{\omega_{n,m}}(x_{m}) = x_n$. Therefore equation \eqref{eq: lambda 0 first} can be rewritten as
  \begin{equation}\label{eq: lambda 0 second}
  	\lambda(x_0) \leq {\rm Sum}(\omega_{0,n},x_{n})\odot   \lambda(x_n) \odot  \varepsilon,\,\,\forall\,n\geq 1.
  \end{equation}

\textit{Claim:} Any point $z$ which is an accumulation point of the sequence $(x_n)_{n\geq 0}$ belongs to the Aubry set $\Omega$.

Indeed,
let $(x_{n_i})$ be a subsequence of $(x_n)_{n\geq 0}$ which converges to $z$.  Given $\eta >0$ there exists $N_0$ such that $|\sum_{n\geq N_0} q_{j_{n-1}}(x_n)|<\eta$ (because from inequality \eqref{eq: lambda 0} the series $\sum_{n\geq 1} q_{j_{n-1}}(x_n)$ converges). Then for $n_i>N_0$, as $q\leq 0$, we have $\sum_{n_i< n \leq n_{i+1}} q_{j_{n-1}}(x_n)>-\eta.$
As $x_n = \phi_{\omega_{n,n_{i+1}}}(x_{n_{i+1}})=\phi_{j_n}\circ...\circ\phi_{j_{(n_{i+1})-1}}(x_{n_{i+1}})$
 we obtain
\[{\rm Sum}(\omega_{n_i,n_{i+1}},x_{n_{i+1}}) =\sum_{n_i< n \leq n_{i+1}} q_{j_{n-1}}(\phi_{j_n}\circ...\circ\phi_{j_{(n_{i+1})-1}}(x_{n_{i+1}})) >-\eta.\]
From Lemma \ref{lemma: Sum lipschitz},
\[{\rm Sum}(\omega_{n_i,n_{i+1}},z) >-\eta-\frac{C}{1-\gamma}\cdot d(x_{n_{i+1}},z).\]
Furthermore,
\[d(z, \omega_{n_i,n_{i+1}}(z))\leq d(z, x_{n_i})+d(x_{n_i},\omega_{n_i,n_{i+1}}(z))=d(z, x_{n_i})+d(\omega_{n_i,n_{i+1}}(x_{n_{i+1}}),\omega_{n_i,n_{i+1}}(z))\]
\[\leq d(z, x_{n_i})+\gamma^{(n_{i+1}-n_i)}d(x_{n_{i+1}},z) \stackrel{n_i\to +\infty}{\to} 0.\]
Making $n_i \to \infty$, by definition of $S(\cdot,\cdot)$ we get
$$S(z,z) \geq \limsup_{n_i\to+\infty}{\rm Sum}(\omega_{n_i,n_{i+1}},z)= -\eta.$$
 As $\eta$ is any positive number we conclude that $S(z,z) = 0$.
This proves the claim.

\bigskip

Now we complete the proof. Let $z$ be an accumulation point of $x_n$ and suppose that $x_{n_i}$ is a subsequence which converges to $z$. From above claim $z\in \Omega$ and we want to show that
\[\lambda(x_0) \leq  S(x_0,z) \odot   \lambda(z)\odot  \varepsilon.\]
With this objective we return to inequality \eqref{eq: lambda 0 second} and apply Lemma \ref{lemma: Sum lipschitz} in order to obtain
\[\lambda(x_0) \leq {\rm Sum}(\omega_{0,n_i},x_{n_i})\odot   \lambda(x_{n_i}) \odot  \varepsilon  \leq {\rm Sum}(\omega_{0,n_i},z) \odot   \frac{C}{1-\gamma} d(x_{n_i},z)\odot   \lambda(x_{n_i}) \odot  \varepsilon\]
Making $n_i \to \infty$ (remember that $\lambda$ is u.s.c.) we have
\[\lambda(x_0) \leq \liminf_{n_i\to\infty} {\rm Sum}(\omega_{0,n_i},z) \odot   \lambda(z) \odot  \varepsilon.\]
Therefore, as we also have $\phi_{\omega_{0,n_i}}(z) \to x_0$, we obtain
\[\lambda(x_0) \leq S(x_0,z) \odot   \lambda(z) \odot   \varepsilon.\]
\end{proof}

\section{Characterization of the invariant probability for non place dependent mpIFS}\label{sec: non dependence}

 In this section we propose to prove that for a compact space $J$ and a mpIFS $S=(X,\phi,q)$, if  $q_j(x)=q_j$ does not depend on $x$, then there exists a unique invariant idempotent  probability for $M_{\phi,q}$. Furthermore we present explicitly its density (see Theorem \ref{thm:represent idemp invar cte}). This result somehow generalize the ones for a finite number of maps as founded in \cite{MZ}, \cite{dCOS20fuzzy} and \cite{ExistInvIdempotent}.

We start by proving some auxiliary lemmas. Initially we remark that for any $\omega =(j_1,...,j_n) \in J^{n}$ and $x\in X$ we have
\[{\rm Sum}(\omega,x) = q_{j_1}+q_{j_2}+...+q_{j_n}.\]
 In this section we consider also the space ${J}^{\mathbb{N}}$, which is compact with the metric
\[\tilde{d}((i_1,i_2,i_3,...),(j_1,j_2,j_3,...)) = \sum_{n\geq 1} \gamma^n d_J(i_n,j_n),\]
where $\gamma<1$ satisfies \eqref{eq: gamma contraction}.	
   	
The next lemma and definition are well known in the literature with more or less generality (finite symbols in $J$), see \cite[Theorem 2.3]{Lewe93}, \cite[Proposition 1.3]{FanRuelle}, \cite[Theorem 2.1]{MihShift} and  \cite[Theorem 1.5]{StrSwa}.

\begin{lemma}\label{lemma:x_omega} Given any $\omega=(j_1,j_2,...) \in J^{\mathbb{N}}$, there exists a unique point $x_{\omega} \in X$ satisfying
\[x_{\omega} = \lim_{n\to+\infty} \phi_{j_1}\circ...\circ\phi_{j_n}(x_n) \]
for any sequence $(x_n)$ of points of $X$.
\end{lemma}
\begin{proof}
If $A\subseteq X$ is a compact set, as $\phi_j$ is continuous, we get that $\phi_j(A)$ is compact. Then we have a sequence of compact sets
\[X \supseteq \phi_{j_1}(X) \supseteq \phi_{j_1}\circ\phi_{j_2}(X) \supseteq...\supseteq \phi_{j_1}\circ...\circ\phi_{j_n}(X)\supseteq...\]
which therefore has a  nonempty compact intersection. Furthermore, as ${\rm diam}(\phi_j(A))\leq \gamma\cdot {\rm diam}(A)$, there exists a unique point $x_{\omega}$ such that
\[\{x_{\omega}\} = \bigcap_{n\geq 1} \phi_{j_1}\circ...\circ\phi_{j_n}(X).\]
As $x_{\omega}$ and $\phi_{j_1}\circ...\circ\phi_{j_n}(x_n)$ are points in $\phi_{j_1}\circ...\circ\phi_{j_n}(X)$ and ${\rm diam}(\phi_{j_1}\circ...\circ\phi_{j_n}(X))\leq \gamma^{n}\cdot {\rm diam}(X)$ we conclude the proof.
\end{proof}

\begin{definition}\label{def:pi_omega} We denote by $\pi:J^{\mathbb{N}}\to X$ the map given by $\pi(\omega) = x_{\omega}$, where $x_{\omega}$ is defined in Lemma \ref{lemma:x_omega}.
\end{definition}

\begin{lemma}
	The map $\pi:J^{\mathbb{N}}\to X$ in Definition \ref{def:pi_omega}  is continuous. More precisely
	\[d_X(\pi(\omega),\pi(\eta)) \leq \tilde{d}(\omega,\eta),\,\,\,\forall \,\omega,\eta \in J^{\mathbb{N}}.\]
\end{lemma}
\begin{proof}
Given $\omega=(i_1,i_2,i_3,...)$ and $\eta = (j_1,j_2,j_3,...)$ and any point $x\in X$ we have $\pi(\omega) = \lim_n \phi_{i_1}\circ \dots \circ \phi_{i_n}(x)$ and $\pi(\eta) = \lim_n \phi_{j_1}\circ \dots \circ \phi_{j_n}(x)$. Then we have, 	
\[d_X(\pi(\omega),\pi(\eta)) =\lim_{n\to+\infty} d_X( \phi_{i_1}\circ \dots \circ \phi_{i_n}(x),\phi_{j_1}\circ \dots \circ \phi_{j_n}(x)).\]
As $\phi$ satisfies $\eqref{eq: gamma contraction}$ we have
\[d_X(\pi(\omega),\pi(\eta)) \leq \lim_{n\to+\infty} \gamma [d_J(i_1,j_1)+d_X( \phi_{i_2}\circ \dots \circ \phi_{i_n}(x),\phi_{j_2}\circ \dots \circ \phi_{j_n}(x))]\]
\[\leq \lim_{n\to+\infty} [\gamma d_J(i_1,j_1)+\gamma^2 d_J(i_2,j_2) +\gamma^2 d_X( \phi_{i_3}\circ \dots \circ \phi_{i_n}(x),\phi_{j_3}\circ \dots \circ \phi_{j_n}(x))]\]
\[\leq.. .\leq \lim_{n\to+\infty} [\gamma d_J(i_1,j_1)+\gamma^2 d_J(i_2,j_2)+...+\gamma^{n}d_J(i_n,j_n)] = \tilde{d}(\omega,\eta).\]
\end{proof}

\begin{definition} We denote by $J_0:=\{j\in J\,|\,q_j=0\}$ which from \eqref{eq: q normalized} is a non-empty set.
\end{definition}

\begin{lemma}\label{lemma:Omega and J_0}
Given a mpIFS, $S=(X,\phi,q)$, if  $q_j(x)=q_j$ does not depend of $x$, then the Aubry set satisfies $\Omega=\pi(J_0^{\mathbb{N}})$.
\end{lemma}
\begin{proof} If $z\in\pi(J_0^{\mathbb{N}})$ and we write $z=\pi(i_1,i_2,i_3,...)$,  then $z=\lim_{n\to\infty}\phi_{i_1}\circ \ldots \circ\phi_{i_n}(z)$ and then $S(z,z) \geq \lim_{n \to \infty} q_{i_1}+...+q_{i_n} = 0$, which proves that $z\in \Omega$.

On the other hand, if $z\in \Omega$, then there exists a sequence of points having the form $z_n =\phi_{j_1^n}\circ \ldots \circ\phi_{j_{k(n)}^n}(z)$ such that $z_n \to z$ and such that
\[0\geq q_{j_1^n}+...+q_{j_{k(n)}^n} >-\frac{1}{n}.\]
From now on we consider two cases to complete the proof.

\textit{Case 1:} there exists a number $M$ such that $k(n)=M$ for infinite indexes $n$. Then taking a subsequence given from such indexes we can suppose that $z_n =\phi_{j_1^n}\circ \ldots \circ\phi_{j_{M}^n}(z)$ for all $n$. By taking a subsequence again we can suppose there exists $(j_1,...,j_M)$ such that $(j_1^n,...,j_M^n)\to (j_1,...,j_M)$ in the space $J^{M}$. As by \eqref{eq: gamma contraction} the map $\phi$ is continuous also in the $j$ variable, we get $z =\phi_{j_1}\circ \ldots \circ\phi_{j_{M}}(z)$ and as $q$ is continuous in $J$ we also get
\[q_{j_1}+...+q_{j_M} = \lim_{n\to\infty}q_{j_1^n}+...+q_{j_M^n} = 0.\]
This shows that $j_1,...,j_M \in J_0$ and as $z =\phi_{j_1} \circ \ldots \circ \phi_{j_{M}}(z)=\phi_{j_1} \circ \ldots \circ \phi_{j_{M}}\phi_{j_1} \circ \ldots \circ \phi_{j_{M}}(z)=...$ we get $z = \pi(j_1,...,j_M,j_1,...,j_M,j_1,...,j_M,...)$.

\textit{Case 2:} for any number $M$ we have $k(n)>M$ for sufficiently large $n$. Then, taking a subsequence, we can suppose that $k(n)$ is increasing and $k(n)>n$. Fixed any $\tilde{j}\in J$ we denote by $\omega^n=(j_1^n,j_2^n,...,j_{k(n)}^n,\tilde{j},\tilde{j},\tilde{j},...)$ in $J^{\mathbb{N}}$. Such sequence has a convergent subsequence $\omega^{n_i}$ which we can assume to be the initial sequence $\omega^n$. Then, there exists a point $\omega = (j_1,j_2,...)\in J^{\mathbb{N}}$ such that $\omega^n \to \omega$. We have $0\geq q_{j_1} = \lim_{n \to \infty} q_{j_1}^n \geq\lim_{n \to \infty} [q_{j_1^n}+...+q_{j_{k(n)}^n}] \geq -\lim_{n \to \infty} \frac{1}{n} =0$, and so $j_1 \in J_0$. The same argument can be applied to $j_2,j_3,...$ and so we get  $\omega \in J_0^{\mathbb{N}}$.

Now we will prove that $z= \pi(\omega)$. With this purpose we will prove that for any $m$ there exists a point $y_m \in X$ such that $z = \lim_{m}\phi_{j_1} \circ ...\circ \phi_{j_m}(y_m)$.  Let us call $y_m := \phi_{j_{m+1}^m} \circ ...\circ \phi_{j_{k(m)}^m}(z)$. Then we get
\[d_X(z, \phi_{j_1} \circ ...\circ \phi_{j_m}(y_m)) \leq d_X(z, \phi_{j_1^m} \circ ...\circ \phi_{j_m^m}(y_m))+d_X(\phi_{j_1^m} \circ ...\circ \phi_{j_m^m}(y_m),\phi_{j_1} \circ ...\circ \phi_{j_m}(y_m))\]
\[= d_X(z, z_m) +  d_X(\phi_{j_1^m} \circ ...\circ \phi_{j_m^m}(y_m),\phi_{j_1} \circ ...\circ \phi_{j_m}(y_m))\]
\[\stackrel{\eqref{eq: gamma contraction}}{\leq} d_X(z, z_m) +\gamma [d_J(j_1^m,j_1)+d_X(\phi_{j_2^m} \circ ...\circ \phi_{j_m^m}(y_m),\phi_{j_2} \circ ...\circ \phi_{j_m}(y_m))]\]
\[\leq d_X(z, z_m) +\gamma [d_J(j_1^m,j_1)]+\gamma^2[d_J(j_2^m,j_2)+d_X(\phi_{j_3^m} \circ ...\circ \phi_{j_m^m}(y_m),\phi_{j_3} \circ ...\circ \phi_{j_m}(y_m))]\]
\[\leq ...\leq d_X(z, z_m) +  \sum_{k=1}^{m}\gamma^{k} d_J(j_k^m,j_k) \leq d_X(z, z_m)+ \tilde{d}(\omega^m,\omega). \]
Therefore, as $z_m\to z$ and $\omega_n\to\omega$, we get $z= \lim_{m \to \infty}\phi_{j_1} \circ ...\circ \phi_{j_m}(y_m)$ and so $z = \pi(\omega)$.

\end{proof}

\begin{lemma}\label{lemma:S zero in Omega}	Given a mpIFS $S=(X,\phi,q)$ such that $q_j(x)=q_j$ does not depend of $x$, we have $S(z,y) = 0$ for any  $z\in\Omega$ and $y\in X$.
\end{lemma}

\begin{proof} By applying Lemma \ref{lemma:Omega and J_0}, we conclude that there exists $\omega = (j_1,j_2,j_3,...) \in J_0^{\mathbb{N}}$ such that $z =\pi(\omega)$. Then, we have
$z = \lim_{n \to \infty} \phi_{j_1}\circ...\circ\phi_{j_n}(y)$ and
\[0\geq S(z,y) \geq \lim_{n \to \infty} q_{j_1}+...+q_{j_n} = \lim_{n \to \infty} 0 = 0.\]
\end{proof}

Now we present the main result of this section.

\begin{theorem}\label{thm:represent idemp invar cte}  Given a mpIFS, $S=(X,\phi,q)$, if  $q_j(x)=q_j$ does not depend on $x$, then there exists a unique invariant idempotent  probability for $M_{\phi,q}$. Let $\lambda \in U(X,\mathbb{R}_{\max})$ be its density. Then for any $z \in \Omega$ we have
	\[\lambda(x) = S(x,z) = \bigoplus_{(j_1,j_2,j_3,...)\in\pi^{-1}(x)}[q_{j_1}+q_{j_2}+q_{j_3}+...].\]
\end{theorem}

\begin{proof}  Let $\mu\in I(X)$ be invariant for $M_{\phi,q}$ and $\lambda$ be its u.s.c. density given by Theorem \ref{teo : densidade Maslov}. Initially we will prove that $\lambda(x) = S(x,z)$ for any $z\in \Omega$. By applying Theorem \ref{teo: main result} we have
	\[\lambda(x) = \bigoplus_{z\in \Omega} [S(x,z)\odot \lambda(z)], \,\,\,\forall\,x\in X.\]
We recall that $S\leq 0$ and $\bigoplus_{x\in X} \lambda(x)=0$ what means that $\bigoplus_{z\in\Omega} \lambda(z) = 0$.
It follows that, for any point $\tilde{z}\in \Omega$, by applying Lemma \ref{lemma:S zero in Omega},  \[\lambda(\tilde{z}) = \bigoplus_{z\in \Omega} [S(\tilde{z},z) \odot \lambda(z)] =\bigoplus_{z\in \Omega} [0 \odot \lambda(z)] =0 .\]
Therefore $\lambda(z) =0$ for any $z\in \Omega$ and consequently
\[\lambda(x) = \bigoplus_{z\in \Omega} S(x,z).\]
Furthermore, if $z_1$ and $z_2$ are points of $\Omega$ we have $S(x,z_1)=S(x,z_2)$ because
\[S(x,z_1) = S(x,z_1) \odot 0 = S(x,z_1)\odot S(z_1,z_2) \leq S(x,z_2) \]
and
\[ S(x,z_2) = S(x,z_2) \odot 0 = S(x,z_2)\odot S(z_2,z_1) \leq S(x,z_1). \]
This shows that $S(x,z)$ does not depend of the point $z\in \Omega$. Therefore  $$\lambda(x) = \bigoplus_{z\in \Omega} S(x,z)=S(x,\tilde{z}), \,\forall \tilde{z}\in \Omega.$$

From now on we will prove that for any $z\in \Omega$ we have
	\[S(x,z) =\bigoplus_{(j_1,j_2,j_3,...) \in \pi^{-1}(x)}[q_{j_1}+q_{j_2}+q_{j_3}+...].\]

Initially we will prove that $\bigoplus_{(j_1,j_2,j_3,...) \in \pi^{-1}(x)}[q_{j_1}+q_{j_2}+q_{j_3}+...]\leq S(x,z)$. In this way we just need to consider the case $\pi^{-1}(x)\neq \varnothing$. If for some $(j_1,j_2,j_3,...)\in J^{\mathbb{N}}$ we have $\pi(j_1,j_2,j_3,...)=x$  then $x = \lim_{k\to\infty} \phi_{j_1}\circ...\circ\phi_{j_k}(z)$ and by definition of $S(x,z)$ we get
\[S(x,z) \geq \lim_{k\to\infty} q_{j_1}+...+q_{j_k} = q_{j_1}+q_{j_2}+q_{j_3}+...\]
Therefore, by taking a supremum over such sequences $(j_1,j_2,j_3,...)$, we get
$$S(x,z)\geq \bigoplus_{(j_1,j_2,j_3,...) \in \pi^{-1}(x)}[q_{j_1}+q_{j_2}+q_{j_3}+...]$$
as claimed.

Now we will prove that $S(x,z)\leq 	\bigoplus_{(j_1,j_2,j_3,...) \in \pi^{-1}(x)}[q_{j_1}+q_{j_2}+q_{j_3}+...]$. With this purpose, we can suppose $S(x,z)\neq -\infty$. As $z\in \Omega$, there exists $\omega \in J_0^{\mathbb{N}}$, $\omega = (i_1,i_2,i_3,...)$ such that $z = \pi(\omega)$. As $S(x,z) \neq -\infty$ we have $S_{\epsilon}(x,z)\neq -\infty$ for any $\epsilon>0$. By definition of $S_{\epsilon}$, for each $\epsilon=\frac{1}{n}$, there exists a sequence $(j_1^n,...,j_{k(n)}^n)$ such that $d(x, \phi_{j_1^n} \circ\dots\circ \phi_{j_{k(n)}^n}(z))<\frac{1}{n}$ and
\[-\infty < S_{\frac{1}{n}}(x,z) \leq q_{j_1^n}+...+q_{j_{k(n)}^n}+\frac{1}{n}.\]
We consider the sequence $(\omega_n)_{n\geq 1}$ in $J^{\mathbb{N}}$  defined by $\omega_n= (j_1^n,...,j_{k(n)}^n, i_1,i_2,i_3,...)$ which has a convergent subsequence (we can suppose the sequence $(\omega_n)$ converges). Let $\omega_\infty=(l_1,l_2,l_3,...)$ be the limit of $(\omega_n)$. As $d(x, \phi_{j_1^n} \circ\dots\circ \phi_{j_{k(n)}^n}(z))<\frac{1}{n}$ and $\phi_{j_1^n} \circ\dots\circ \phi_{j_{k(n)}^n}(z)=\pi(\omega_n)$ we get $d(x,\pi(\omega_n))<\frac{1}{n}$ and so, as $\pi$ is continuous, $x=\pi(\omega_\infty)$. Furthermore
\[S(x,z) = \lim_{n\to\infty} S_{\frac{1}{n}}(x,z) \leq \lim_{n\to\infty} q_{j_1^n}+...+q_{j_{k(n)}^n} = \lim_{n\to\infty} q_{j_1^n}+...+q_{j_{k(n)}^n}+q_{i_1}+q_{i_2}+q_{i_3}+...\]
As $q$ is continuous we get that, for any fixed $N\geq 1$,
\[S(x,z)\leq \lim_{n\to\infty} q_{j_1^n}+...+q_{j_{k(n)}^n}+q_{i_1}+q_{i_2}+q_{i_3}+....\leq q_{l_1}+...+q_{l_N}\]
By taking a limit in $N$ we get
\[S(x,z) \leq q_{l_1}+q_{l_2}+q_{l_3}+...\]
and so $S(x,z)\leq \bigoplus_{(j_1,j_2,j_3,...) \in \pi^{-1}(x)}[q_{j_1}+q_{j_2}+q_{j_3}+...]$.
\end{proof}

\section{Applications to fuzzy IFS}\label{sec:Applications to fuzzy IFS}

In this section we suppose that $J$ is a finite set. Our main references for fuzzy IFS (IFZS for short) are \cite{Cab92}, where IFZS were introduced as a new tool for the inverse problem of fractal or image construction, and \cite{Oli17} where this theory was extended for generalized IFSs in the sense of \cite{MR2415407}. A fuzzy subset of $X$ is any function  $u: X \to [0,1]$. The family of fuzzy subsets of $X$ is denoted by $\mathcal{F}_{X}$, that is $\mathcal{F}_{X}:=\{ u \; | \; u: X \to [0,1]\}.$
\begin{definition}\label{grey level} \emph{Given $\alpha \in (0,1]$ and $u \in \mathcal{F}_{X}$, }the grey level \emph{or} $\alpha$-cut of $u$\emph{ is the set
$$[u]^{\alpha}:=\{x \in X \; | \; u(x)\geq \alpha \},$$
that is, the set of points where the grey level exceeds the threshold value $\alpha$.
For $\alpha=0$ we define
$$[u]^{0}:={\operatorname{supp}(u):=} \overline{\bigcup\{[u]^{\alpha} \; | \;  \alpha >0\}}=\overline{\{x\in X:u(x)>0\}}.$$}
\end{definition}

A fuzzy set $u \in \mathcal{F}_{X}$ is normal if there is $x \in X$ such that $u(x)=1$ and it is compactly supported  if $[u]^0$ is compact. To built an IFS theory we need to restrict $\mathcal{F}_{X}$ to a smaller family,
$$\mathcal{F}_{X}^* := \{ u \in \mathcal{F}_{X} \; | u \text{ is normal, u.s.c. and compactly supported}\}.$$

We can define a distance $d_\infty$ in $\mathcal{F}_{X}^*$ by
$$d_\infty (u,v) := \bigoplus_{\alpha \in [0,1]} h([u]^{\alpha},[v]^{\alpha}),$$
for $u, v \in \mathcal{F}_{X}^*$, where $h$ is the Hausdorff distance. It is known that $d_\infty$ is a metric (see~\cite{Dia94}), which is complete provided $X$ is compact (see~\cite{Cab92}).

Given a map $\phi: X \to X$, $u \in \mathcal{F}_{X}$, we define a new fuzzy set $\phi(u) \in \mathcal{F}_{X}$ as follows (Zadeh's Extension Principle  \cite{Zad}):
$$\phi(u)(x) :=
\left\{
  \begin{array}{ll}
    \bigoplus_{y\in \phi^{-1}(x)} u(y), & if \; x \in \phi(X); \\
    0, & \mathrm{otherwise.}
  \end{array}
\right.
$$
\begin{definition}
A grey level map is a nonzero function $\rho: [0,1] \to [0,1]$. We said that a grey level map satisfy ndrc condition  or is a ndrc map, if\\
a) $\rho$ is nondecreasing;\\
b) $\rho$ is right continuous.
\end{definition}

Consider $J:=\{1,\ldots,n\}$ and the max operator $a\vee b:=\max\{a,b\}$.
\begin{definition}\label{def: admissible grey level maps}  \emph{A
 system of grey level maps $(\rho_j)_{j \in J}: [0,1]\to [0,1]$ is }admissible\emph{ if it satisfies all the conditions\\
a) $\rho_j$ is nondecreasing;\\
b) $\rho_j$ is right continuous;\\
c) $\rho_j(0)=0$;\\
d) $\rho_j(1)=1$ for some $j$.}
\end{definition}

\begin{definition}\label{def: ifzs definition} \emph{ Let $\mathcal{R}=(X, (\phi_j)_{j \in J})$ be an IFS and $(\rho_j)_{j \in J}$ be an admissible system of grey level maps. Then the system $\mathcal{Z_R}:=(X, (\phi_j)_{j\in J}, (\rho_j)_{j\in J})$ is called} an iterated fuzzy function system (IFZS in short). Inspired by the (HB) operator, we define the Fuzzy Hutchinson-Barnsley (FHB) operator associated to $\mathcal{Z_R}$ by
$$\mathcal{Z_R}(u):= \bigvee_{j \in J} \rho_j(\phi_j(u))$$
for all $u\in\mathcal{F}^*_X$.
\end{definition}
A more general version of the next theorem can be founded in \cite[Theorem 3.15]{Oli17} for Matkowski contractive IFZS.
\begin{theorem}\cite[Theorem 2.4.1]{Cab92}\label{thm: cabrelli fixed point fractal operator}
 Given a contractive IFZS $\mathcal{Z_R}=(X, (\phi_j)_{j \in J})$,  the FHB operator $\mathcal{Z_R}: \mathcal{F}_{X}^* \to \mathcal{F}_{X}^*$ is a Banach contraction in $(\mathcal{F}_{X}^*, d_\infty)$. More precisely,
$$d_\infty(\mathcal{Z_R}(u),\mathcal{Z_R}(v)) \leq \lambda \; d_\infty(u,v), \; \forall u,v \in \mathcal{F}_{X}^*,$$
where $\lambda:=\max\{Lip(\phi_j):j \in J\}$ and $Lip(\phi_j)$, $j \in J$, are contraction constants of $\phi_js$, respectively.
In particular, if $X$ is complete, then there exists a unique $u \in \mathcal{F}_{X}^*$ such that $\mathcal{Z_R}(u)=u$ and, moreover, for any $v \in \mathcal{F}_{X}^*$ we get
$d_\infty(\mathcal{Z_R}^{(k)}(v), \; u) \to 0,$
where $\mathcal{Z_R}^{(k)}(v)$ denotes the $k$-th iteration of the (FHB) operator $\mathcal{Z_R}$.
\end{theorem}
\begin{definition}\emph{The fuzzy set $u$ from the Theorem \ref{thm: cabrelli fixed point fractal operator} is called} the fuzzy attractor or fuzzy fractal generated by IFZS $\mathcal{Z_R}$.
\end{definition}
There are several ways to introduce topologies on $I(X)$. Canonically (see \cite{Rep10}, \cite{Zar} and \cite{ZAI}), we endow $I(X)$ with the topology $\tau_p$ of the pointwise convergence, that is, the basis of the topology $\tau_p$ consists of sets of the form
$$
\{\mu\in I(X):|\mu(\varphi_1)-\mu_1(\varphi_1)|<\varepsilon,...,|\mu(\vp_n)-\mu_n(\vp_n)|<\varepsilon\}{,}
$$
where $\varepsilon>0$, $n\in\N$, $\mu_1,...,\mu_n\in I(X)$ and $\vp_1,...,\vp_n\in C(X, \mathbb{R})  $. Equivalently, $\tau_p$ is the topology on $C(X, \mathbb{R})  $ induced by the Tychonoff product topology on $\R^X$.\\
Clearly, for $(\mu_n)\subset I(X)$ and $\mu\in I(X)$, we have that
$\mu_n \overset{\tau_p}{\to} \mu$ in $I(X)$ {if and only if} $\mu_n(\psi) \to \mu(\psi)$ for all $\psi \in C(X, \mathbb{R})  $.

From \cite{dCOS20fuzzy} we know that given the scale map $\theta: [-\infty,\; 0] \to [0,1]$ defined by $\theta(t)=e^t$ (there are other possible choices), it induces a bijection $\Theta$ between $I(X)$ and $\mathcal{F}_{X}^* $ given by
$$u(x)=\Theta(\mu)(x)=\theta(\lambda_{\mu}(x)),$$
for any $\mu \in I(X)$ with density $\lambda_{\mu}$.

Moreover, given a mpIFS $\mathcal{S}=(X, (\phi_j)_{j\in J}, (q_j)_{j\in J})$, where $\phi_{j}:X \to X$ and $q_{j} \in [-\infty, 0]$ such that $\bigoplus_{j \in J} q_{j}=0$, by \cite{dCOS20fuzzy} we know that, given the associated IFZS $\mathcal{Z_R}=(X,(\phi_j)_{j\in J},(\rho_j)_{j\in J})$, where the grey level functions are
$$\rho_j(t):=\theta(q_j + \theta^{-1}(t)),\; t \in [0,1],$$
we have,
$$\Theta \circ M_{\phi ,q} = \mathcal{Z_R}\circ \Theta.$$
In particular $M_{\phi ,q}(\mu)=\mu$ if, and only if, $\mathcal{Z_R}(u)=u$ for $u=\Theta(\mu)$.

Then, we can introduce a metric $d_{\theta}$ induced by $\Theta$ from the metric space of fuzzy sets $(\mathcal{F}_{X}^* ,d_{\on{f}})$, given by $d_{\theta}(\mu, \nu)= {d_{\on{f}}}(\Theta(\mu), \Theta(\nu)),$ in such way that the spaces $(I(X),d_\theta)$ and $(\mathcal{F}_{X}^* ,d_{\on{f}})$ are homeomorphic. By \cite[Theorem 3.5]{dCOS20fuzzy}  is complete, since $(X,d)$ is compact:
\begin{lemma}\cite[Lemma 5.7]{dCOS20fuzzy}\label{filipmetric}
For every  $\mu=\bigoplus_{x\in X}\lambda(x)\odot\delta_x,\;\nu=\bigoplus_{x\in X}\eta(x)\odot\delta_x\in I(X)$, we have
$$
d_\theta(\mu,\nu)=\bigoplus_{\beta\in(-\infty,0]}h(\{x\in X:\lambda(x)\geq \beta\},\{x\in X:\eta(x)\geq \beta\}),
$$
where $h$ is the Hausdorff distance.
\end{lemma}

\begin{proposition}\cite[Proposition 5.8]{dCOS20fuzzy}\label{fil11}
The metric space $(I(X),\, d_\theta)$ is complete and the topology $\tau_\theta$ induced by $d_\theta$ is finer than the  pointwise convergence topology $\tau_p$. In other words, $\tau_p\subset\tau_\theta$.
\end{proposition}

As pointed in \cite[Example 5.9]{dCOS20fuzzy}, $(I(X),\, d_\theta)$ may not be compact. Although, the topological space $(I(X),\,  \tau_p)$ is compact \cite[Theorem 5.3]{Rep10}.

\begin{proposition}\label{prop:appl fuzzy}
    Let $\mathcal{S}=(X, (\phi_j)_{j\in J}, (q_j)_{j\in J})$ be a mpIFS, where $q_j(x)=q_j$ does not depend of $x$, and the associated IFZS
    $$\mathcal{Z_R}=(X,(\phi_j)_{j\in J},(\rho_j)_{j\in J}),$$
    where $\rho_j=e^{q_j}t$ and $\theta$ is a scale function, for instance $\theta(t)=e^t$ (that is, $\theta(q_j + \theta^{-1}(t))=e^{q_j}t$), then the unique fuzzy attractor of $\mathcal{Z_R}$ satisfies
    $$u(x)= \bigoplus_{(j_1,j_2,j_3,...) \in \pi^{-1}(x)}e^{(q_1+q_2+q_3+...)}.$$
\end{proposition}
\begin{proof}
  From Proposition~\ref{thm:represent idemp invar cte} we know that the density $\lambda$ of the unique invariant measure $\mu$ is given by $\lambda(x) = \bigoplus_{(j_1,j_2,j_3,...) \in \pi^{-1}(x)}(q_1+q_2+q_3+...)$.
  Since $M_{\phi ,q}(\mu)=\mu$ if, and only if, $\mathcal{Z_R}(u)=u$ for $u=\Theta(\mu)=\theta(\lambda(x))=e^{\lambda(x)}$ the result follows.
\end{proof}

\begin{remark}
   Proposition~\ref{prop:appl fuzzy} gives an alternative way to obtain a representation of the fuzzy fractals, when the grey level function has the form $\rho_j=\theta(q_j + \theta^{-1}(t)), \;j \in J$. Theorem~\ref{thm: cabrelli fixed point fractal operator} gives only the existence or the approximation, via iteration. Other way is a discrete scheme proved in \cite[Theorem 6.3]{dCOS20fuzzy}.
\end{remark}

\begin{remark}
  In \cite{Cab92} we found the following statement ``It should be mentioned that in some practical treatments of the inverse problem, success has already been achieved by employing more general sets of grey level functions $\varphi_i$ : for example, ``place-dependent" grey level maps $\varphi_i: [0,1] \rightarrow[0,1]$. However, this will be the subject of future work." We were not able to find subsequent developments on fuzzy fractals arising from place-dependent grey level maps. As we now have a complete description of the invariant idempotent probabilities for place-dependent mpIFS, we can associate  to a given a mpIFS $\mathcal{S}=(X,\phi, q)$, where $\phi_{j}:X \to X$ and $q_{j}:X \to (-\infty, 0]$ such that $\bigoplus_{j \in J} q_{j}(x)=0, \forall x \in X$, by the associated IFZS $\mathcal{Z_R}=(X,(\phi_j)_{j\in J},(\rho_j)_{j\in J})$, where the place-dependent grey level functions are given by
$$\rho_j(t,x):=\theta(q_j(x) + \theta^{-1}(t)),\; t \in [0,1],\; x \in X$$
then,
$$\Theta \circ M_{\phi ,q} = \mathcal{Z_R}\circ \Theta.$$
In particular $M_{\phi ,q}(\mu)=\mu$ if, and only if, $\mathcal{Z_R}(u)=u$ for $u=\Theta(\mu)$.

Indeed, taking the scale function $\theta(t)=e^{t}$ we get $$\rho_j(t,x)=\theta(q_j(x) + \theta^{-1}(t))=e^{q_j(x)}t$$
thus, recalling that $u(x)=\Theta (\mu)(x)=\theta(\lambda(x))=e^{\lambda(x)}$, for $\mu=\bigoplus_{x\in X}\lambda(x)\odot\delta_x$, we obtain
$$ \mathcal{Z_R}\circ \Theta (\mu) =\mathcal{Z_R}(u)(x)= \bigvee_{j \in J}\phi_j \rho_j(u(x),x)=\bigvee_{j \in J} \phi_j(e^{q_j}u)(x)=$$
$$= \bigvee_{j \in J} \max_{(j,y)\in \phi^{-1}(x)} e^{q_j(y)} u(y)= \bigoplus_{(j,y)\in \phi^{-1}(x)}e^{q_j(y)}e^{\lambda(y)}=
e^{\bigoplus_{(j,y)\in \phi^{-1}(x)} q_j(y)\odot\lambda(y)}=$$
$$=\theta(L_{\phi,q}(\lambda)(x))=\Theta(M_{\phi ,q}(\mu))(x),$$
where $L_{\phi,q}(\lambda)(x):=\max_{(j,y)\in \phi^{-1}(x)} q_{j}(y) \odot  \lambda(y).$

So we have a full characterization of the fuzzy fractals (probably not attractors, because we do not have uniqueness) associated to the above IFZS using the characterization of the invariant idempotent measure for the correspondent mpIFS.
\end{remark}

\section{Appendix A: Fundamentals of idempotent analysis}\label{sec:Fundamentals of idempotent analysis}

The following exposition is based on the ideas of Kolokol'tsov and  Maslov~\cite{KM89}, who considered a different setting for applications of idempotent analysis. See also \cite{MR1035374}, \cite{Kol97}, \cite{LitMas96}, \cite{MZ}, \cite{Rep10}, \cite{Zar}, \cite{ZAI}, \cite{dCOS20fuzzy} and \cite{ExistInvIdempotent}, among many others for additional references.

\subsection{Proof of Theorem \ref{teo : densidade Maslov}}
The proof of Theorem \ref{teo : densidade Maslov} will be divided in several parts as exposed below.

\begin{proposition}
  Any $m\in C^*(X,\mathbb{R})$ is order preserving, that is, if $f \leq g$ then $m( f) \leq m( g)$. Moreover
  $$\min_{X} f  \leq m(f) - m(0) \leq \max_{X} f.$$
\end{proposition}
\begin{proof}
  If $f \leq g$ then $f \oplus g = g$. As  $m$ is max-plus additive we obtain $m(g) = m(f \oplus g)=m( f) \oplus m( g)$, thus $m( f) \leq m( g)$.

  As we have a function $m: C(X,\mathbb{R}) \to \mathbb{R}$, then we get $m(0)\neq -\infty$. As for any $f \in C(X, \mathbb{R})$, $\min_{X} f \leq f(x) \leq \max_{X} f$, we obtain
  $$m(\min_{X} f ) \leq m(f) \leq m(\max_{X} f)$$
  $$m((\min_{X} f)\odot 0) \leq m(f) \leq m((\max_{X} f)\odot 0)$$
  $$(\min_{X} f) \odot m(0) \leq m(f) \leq (\max_{X} f) \odot m( 0).$$
\end{proof}

 The semimodule $\mathcal{V}:=(C(X,\mathbb{R}), \oplus, \odot)$ can be topologized with the usual structure given by the metric $d_{\infty}(f,g)=\bigoplus_{x\in X} |f(x)- g(x)|$.  Introducing a metric $\rho: \mathbb{R}_{\max} \times \mathbb{R}_{\max} \to [0, +\infty)$ given by $\rho(a,b):=|\exp(a)-\exp(b)|$ we obtain a topological semiring $(\mathbb{R}_{\max}, \rho)$ and for the set $\mathcal{F}(X,\mathbb{R}_{\max}):=\{ f: X \to \mathbb{R}_{\max}\}$ we can consider the topology given by  $$d_{\rho}(f,g)=\bigoplus_{x\in X} \rho(f(x), g(x))=\bigoplus_{x\in X} |\exp(f(x))- \exp(g(x))|,$$ for any $f, g \in \mathcal{F}(X,\mathbb{R}_{\max})$.

\begin{proposition}\label{prop: m is cont in d_sup}
  Any $m\in C^*(X,\mathbb{R})$ is nonexpansive with respect to the usual sup-norm $d_{\infty}$ in $C(X, \mathbb{R})$ and the absolute value $|\cdot|$ in $\mathbb{R}$. In particular it is continuous.
\end{proposition}
\begin{proof}
   Indeed,
   $$ -d_{\infty}(f, g) + g(x) \leq f(x) \leq d_{\infty}(f, g) + g(x), \forall x \in X$$
   $$ -d_{\infty}(f, g) + m(g) \leq m(f) \leq d_{\infty}(f, g) + m(g)$$
   $$ |m(f) - m(g)| \leq d_{\infty}(f, g).$$
\end{proof}

\begin{definition} We say that a sequence of functions $f_n \in C(X,\mathbb{R})$ converges pointwise to a function $f \in \mathcal{F}(X,\mathbb{R}_{\max})$ if $\lim_{n\to\infty} \rho(f_n(x),f(x)) =0$, for any $x\in X$. Equivalently:\newline
1. $\lim_{n\to+\infty}f_n(x) = f(x)$, for all $x\in X$ such that $f(x)\in\mathbb{R} $\newline
2. $\lim_{n\to+\infty}f_n(x) = -\infty$, for all $x\in X$ such that $f(x)=-\infty. $
\end{definition}

\begin{definition}\label{def: Dirac function}
   Given $a \in \mathbb{R}$ and $x \in X$, a fixed point, we define the Dirac function
   \[ g^{a}_{x}(y):=
   \left\{
     \begin{array}{ll}
       a, & y=x \\
       -\infty, & y\neq x
     \end{array}
   \right.
   \]
and $\Delta(X, \mathbb{R})$  the set of these functions.
\end{definition}

\begin{lemma}\label{lem: e particular seq}
   There exists a monotone nonincreasing sequence of continuous functions $f_n\in C(X,\mathbb{R})$ which converges pointwise to $g^{a}_{x}$.
\end{lemma}
\begin{proof} A consequence of the fact that $X$ is compact metric space is that we can easily build a sequence in $C(X, \mathbb{R})$ of nonincreasing functions converging pointwise to $g^{a}_{x}$. Indeed, for each $n \in \mathbb{N}$ consider the function $f_{n}$ defined by
   \begin{equation}\label{stand} f_{n}(y):=
   \left\{
     \begin{array}{ll}
       a, & d(y,x)<1/n \\
       -n, & d(y,x)>2/n \\
       a-(n^2 + a n)(d(y,x) - 1/n), & 1/n \leq d(y,x) \leq 2/n
     \end{array}
   \right. .
  \end{equation}
It is obviously continuous in $y$ and monotone nonincreasing in $n$. Moreover, for any $y \in X$, $y\neq x$ we get $d(y,x)>2/n$ for $n$ large enough. In this case we obtain $f_n(y)\to-\infty$.
This proves the pointwise convergence.
\end{proof}

\begin{definition}\label{def: standard continuous approximation}
	Given $g^{a}_{x}$ we will call the sequence of functions $(f_{n})$ given in \eqref{stand} as the \emph{standard continuous approximation of} $g^{a}_{x}$.
\end{definition}

\begin{lemma}
	Let $(f_n)$ be a nonincreasing sequence of continuous functions converging pointwise to a continuous function $f$ in a compact set $X$. Then the convergence is uniform.	
\end{lemma}
\begin{proof}
	Suppose in contradiction there is an $\varepsilon>0$ and a sequence of points $(x_n)\in X$ such that $f_n(x_n)>f(x_n)+\varepsilon,\; \forall\,n\in\mathbb{N}$ .  As $X$ is compact there exists a subsequence $(x_{n_i})$ and a point $x_0\in X$ such that $x_{n_i}\to x_0$. Let $n_j$ be such that $f_{n_j}(x_0)<f(x_0)+\varepsilon/2$ (it exists because we have convergence pointwise). As $f$ and $f_{n_j}$ are continuous in $x_0$ there exists a $\delta>0$ such that $d(x,x_0)<\delta \Rightarrow f_{n_j}(x)<f(x)+\varepsilon$. As $(f_n)$ is nonincreasing we also get $$(d(x,x_0)<\delta,\,n\geq n_j) \Rightarrow f_{n}(x)<f(x)+\varepsilon.$$ This is a contradiction because, by definition of $x_n$ and $x_0$, for $n$ large enough we must have $d(x_n,x_0)<\delta$ and also $f_n(x_n)>f(x_n)+\varepsilon$.	
\end{proof}

\begin{lemma} \label{maximal}
	Let $\bar{C}(X,\mathbb{R})$ be the set of functions $g:X\to\mathbb{R}_{\max}$ such that there exists a nonincreasing sequence of continuous functions $(f_n)$ converging pointwise to $g$. This set is closed with respect to $(\oplus,\odot)$ operations. Given $m \in C^{*}(X, \mathbb{R})$ an idempotent measure, we can extend it to an idempotent measure on $\bar{C}(X,\mathbb{R})$ by $\tilde{m}(g) = \lim_{n\to+\infty}m(f_n)$, where $(f_n)$ is any nonincreasing sequence of continuous functions converging pointwise to $g$.
	
	Furthermore, if $C$ is a set which is closed with respect to $(\oplus,\odot)$ operations, such that $C(X,\mathbb{R})\subseteq C \subseteq \bar{C}(X,\mathbb{R})$ and $\overline{m}$ is an idempotent measure on $C$ which extends $m$, then $\overline{m}\leq \tilde{m}$.
\end{lemma}
\begin{proof} If $g,g'\in \bar{C}(X,\mathbb{R})$ and $f_n\to g,\,f'_n\to g'$ we have $f_n \oplus f'_n \to g\oplus g'$ and $a\odot f_n \to a\odot g$ then, $\bar{C}(X,\mathbb{R}) $ is closed   with respect to $(\oplus,\odot)$ operations.
	
	Consider any sequence of monotonous nonincreasing continuous functions $(f_{n})$ converging pointwise to $g$. We claim that the value
	$\tilde{m}(g):=\lim_{n \to \infty} m(f_{n}) \in \mathbb{R}_{\max}$
	is well defined and furthermore it does not depend on $(f_n)$.
	
	Indeed, as $f_{n} \geq f_{n+1}$ we have $m(f_{n}) \geq m(f_{n+1})$. Thus, there exists the limit $\lim_{n \to \infty} m(f_{n}) \in \mathbb{R}_{\max}$. Let  $(f_{n})$ and $(f'_{n})$ be sequences of monotonous nonincreasing continuous functions converging pointwise to $g$. For any fixed $k$ we define a new sequence $\psi_{n}:=f_{k}\oplus f'_{n},$
so that $\psi_{n}$ converges pointwise to $f_{k} \in C(X,\mathbb{R})$. Applying Lemma \ref{maximal} we get that $\psi_n$ converges uniformly to $f_k$. From Proposition \ref{prop: m is cont in d_sup}  we obtain $m(f_k) = \lim_{n \to \infty} m(\psi_{n})$. Then	$m(f_k)= \lim_{n \to \infty} m(f_k \oplus f'_{n}) \geq \lim_{n \to \infty} m(f'_{n}).$ 	Now we can take the limit on the left hand side obtaining $\lim_{k\to\infty}m(f_k) \geq \lim_{n \to \infty} m(f'_{n}).$	Reversing the role of the sequences we obtain that the limits are equal.

	It is easy to see that $\tilde{m}$ is actually an extension of $m$ because, for $h \in  C(X, \mathbb{R})$ we can take the sequence $f_{n}=h, \forall n$, which is continuous, monotone and not increasing, so $\tilde{m}(h)= \lim_{n \to \infty} m(h)= m(h)$.
	
	Now we prove that $\tilde{m}$ is idempotent. If $g,g'\in \bar{C}(X,\mathbb{R})$ and $f_n\to g,\,f'_n\to g'$ we have
	$$\tilde{m}(g\oplus g') = \lim_{n \to \infty} m(f_n\oplus f'_n) =  \lim_{n \to \infty} [m(f_n)\oplus m(f'_n)] = $$ $$= [\lim_{n \to \infty} m(f_n)]\oplus[\lim_{n \to \infty} m(f'_n)]=\tilde{m}(g)\oplus\tilde{m}(g').$$
	and 	$\tilde{m}(a\odot g) = \lim_{n \to \infty} m(a\odot f_n) =  \lim_{n \to \infty} [a\odot m(f_n)] =  a\odot [\lim_{n \to \infty} m(f_n)]=a\odot \tilde{m}(g).$
	
	Finally, if $C$ is a set closed with respect  to $(\oplus,\odot)$ operations, such that $C(X,\mathbb{R})\subseteq C \subseteq \bar{C}(X,\mathbb{R})$ and $\overline{m}$ is an idempotent measure on $C$ which extends $m$, then given $g\in C$ and any nonincreasing sequence of continuous function $(f_n)$ converging pointwise to $g$ we have $g\leq f_n$. Therefore 	$\overline{m}(g) \leq \overline{m}(f_n) = m(f_n )$ 	and taking the limit in $n$ we get $\overline{m}(g) \leq \tilde{m}(g)$.
\end{proof}	

Let us remember the definition of u.s.c. function.

\begin{definition}\label{def:usc}
	A function $f: X \rightarrow \mathbb{R}_{\max}$ is called upper semi-continuous (u.s.c. for short) at a point $x_0 \in X$ if for every real $c> f\left(x_0\right)$ there exists a neighborhood $U$ of $x_0$ such that $f(x)< c$ for all $x \in U$.
	Equivalently, $f$ is upper semi-continuous at $x_0$ if and only if
	$$
	\limsup_{x \to x_0} f(x) \leq f\left(x_0\right).
	$$
\end{definition}

Next Lemma is inspired in \cite[Lemma 1]{KM89}.
\begin{lemma}\label{lem: represent delta dirac property}
  Given an idempotent measure $m \in C^{*}(X, \mathbb{R})$, consider the extension $\tilde{m}$ from Lemma~\ref{maximal}.
\begin{enumerate}
  \item The map $F:\mathbb{R} \times X \to \mathbb{R}_{\max}$ defined by
$F(a,x)=\tilde{m}(g^{a}_{x})$ is u.s.c. with respect to $x$ and monotonous in $a \in \mathbb{R}$.
  \item $F(a,x)=a \odot F(0,x)$, for any $a \in \mathbb{R}, x \in X$;
  \item $$m(h)=\bigoplus_{x \in X} F(h(x),x) = \bigoplus_{x \in X} F(0,x)\odot h(x),$$
for all $h \in C(X, \mathbb{R})$.
\end{enumerate}
\end{lemma}
\begin{proof}

(1) If $a> a'$ then $g_x^a \geq g_x^{a'}$. So $F(a,x)= \tilde{m}(g_x^a) \geq \tilde{m}(g_x^{a'}) = F(a',x)$.
We now prove that the correspondence $x \to F(a,x)$ is u.s.c. Fix $x_0 \in X$ and take any $c >  F(a,x_0)$.
Recall that $F(a,x_{0})=\tilde{m}(g^{a}_{x_{0}})= \lim_{n \to \infty} m(f_{n})< c$ for $(f_{n})$ the standard continuous approximation of $g^{a}_{x_0}$. So there is $N_{c}$ such that for all $n_{0}> N_{c}$ we have,
$m(f_{n_{0}})< c$. Consider  $x \in U=B_{\frac{1}{n_{0}}}(x_{0})$ and choose $n$ big enough so that  $f'_{n} \leq f_{n_{0}}$ where
$(f'_{n})$ is the standard continuous approximation of $g^{a}_{x}$.

Under this hypothesis we have $m(f'_{n}) \leq  m(f_{n_{0}})$ and taking the limit on the left side we get $\lim_{n \to \infty}m(f'_{n}) \leq  m(f_{n_{0}})$, or equivalently
$F(a,x)  \leq  m(f_{n_{0}})<c.$
This proves that the correspondence $x \to F(a,x)$ is u.s.c.

(2)  $F(a,x)= \tilde{m}(g_x^a) = \tilde{m}(a\odot g_x^{0}) =a \odot \tilde{m}(g_x^{0})=a\odot F(0,x)$.

(3) By applying (2) we get the second equality and then we just need to prove that  $$m(h)=\bigoplus_{x \in X} F(h(x),x)$$
for all $h \in C(X, \mathbb{R})$.   As $g_x^{h(x)}(y) \leq h(y),\, \forall\, x,y\in X$, we obtain $\tilde{m}(g_x^{h(x)})\leq \tilde{m}(h) = m(h)\,\forall\,x\in X$. Then
\[ \bigoplus_{x \in X} F(h(x),x)  = \bigoplus_{x \in X} \tilde{m}(g_x^{h(x)}) \leq \bigoplus_{x \in X} m(h) = m(h). \]

To prove the opposite inequality, we consider any $\varepsilon>0$. Denote by $f_n^x, n\geq 0$ the standard approximation of $g_x^{h(x)}$ as given in \eqref{stand}. For each $n\geq 0$, let $x_n\in X$ be such that  $\displaystyle{\oplus_{x\in X}m(f_n^x) < m(f_n^{x_n})+\varepsilon}$.
As $X$ is compact, there exists a point $\tilde{x}$ and a subsequence $(x_{n_i})$ such that $x_{n_i}\to\tilde{x}$. As $h$ is continuous, there exists $k_0$ such that $h(x)<h(\tilde{x})+\varepsilon$ for any $x\in B(\tilde{x},\frac{1}{k_0}).$ For each natural $k\geq k_0$ there exists $N>k$ such that $B(x_{n_i},\frac{2}{n_i})\subset  B(\tilde{x},\frac{1}{k})$ for any $n_i\geq N$. As $h(x_{n_i})<h(\tilde{x})+\varepsilon$ and $B(x_{n_i},\frac{2}{n_i})\subset B(\tilde{x},\frac{1}{k})$, it follows from  definition \eqref{stand} that $f_{n_i}^{x_{n_i}} \leq f_{k}^{\tilde{x}}+\varepsilon $ for $n_i\geq N$ and consequently,  for $n_i\geq N $ we obtain
\begin{equation}\label{iq1}
	\oplus_{x\in X}m(f_{n_i}^x) < m(f_{n_i}^{x_{n_i}})+\varepsilon< m(f_{k}^{\tilde{x}})+2\varepsilon.
\end{equation} 		
	
As $h$ is continuous and $X$ is compact, the function $h$ is uniformly continuous. Then there exists $\delta$ such that $d(x,y)<\delta \Rightarrow |h(y)-h(x)|<\varepsilon$. Consider a finite cover of $X$ by balls of radius $\frac{1}{n_i}< \delta$, with $X\subset B(z_1,\frac{1}{n_i}) \cup ...\cup B(z_{l_i},\frac{1}{n_i})$. If $y\in B(z_j,\frac{1}{n_i})$ then $h(y) <h(z_j)+\varepsilon$. Furthermore, as from definition \eqref{stand}, $f_{n_i}^{h(z_j)}(x)= h(z_j)$ for any $x\in B(z_j,\frac{1}{n_i})$ we get,
$h \leq \oplus_{z_j}f_{n_i}^{z_j} +\varepsilon.$
Finally, applying \eqref{iq1} we have,
\[m(h) \leq m\left(\oplus_{z_j}f_{n_i}^{z_j} +\varepsilon\right) = \oplus_{z_j}m(f_{n_i}^{z_j})+\varepsilon \leq \oplus_{x\in X}m(f_{n_i}^x)+\varepsilon \leq m(f_{k}^{\tilde{x}})+3\varepsilon.\]
Making $k\to\infty$ we get,
\[m(h) \leq \tilde{m}(g_{\tilde{x}}^{h(\tilde{x})}) +3\varepsilon =F(h(\tilde{x}),\tilde{x}) +3\varepsilon \leq \bigoplus_{x\in X}F(h(x),x) +3\varepsilon.\]
As $\varepsilon$ is arbitrary we conclude the proof.
\end{proof}

The next theorem corresponds to Theorem 1 of  \cite{KM89} in the present setting. It contains the claim in Theorem \ref{teo : densidade Maslov}.

\begin{theorem}\label{thm: charact idempotent measures}
  For each $\lambda \in  U(X, \mathbb{R}_{\max})$ consider the functional $m_{\lambda}: C(X,\mathbb{R}) \to \mathbb{R}_{\max}$ defined by
  $$m_{\lambda}(h):=\bigoplus_{x \in X} \lambda(x) \odot h(x),$$
  for any $h \in C(X,\mathbb{R})$. Then,
  \begin{enumerate}
    \item The map $\gamma:U(X, \mathbb{R}_{\max}) \to C^{*}(X,\mathbb{R})$ defined by $\gamma(\lambda)= m_{\lambda}$ is a max-plus isomorphism between $U(X, \mathbb{R}_{\max})$ and $C^{*}(X,\mathbb{R})$;
    \item The function $\lambda \in U(X, \mathbb{R}_{\max})$ is always bounded from above by $m{_\lambda}(0)$ and $m_\lambda \in I(X)$ is equivalent to $\max_{x \in X} \lambda(x)=0$.
      \end{enumerate}
\end{theorem}
\begin{proof}
(1)
Given $\lambda \in U(X, \mathbb{R}_{\max}) $ let us prove that $m_\lambda \in C^{*}(X,\mathbb{R})$. We have
\[m_\lambda(a\odot f) = \bigoplus_{x \in X} \lambda(x) \odot a\odot f(x) = a\odot \bigoplus_{x \in X} \lambda(x) \odot f(x) = a\odot m_\lambda( f)\]
and
\[m_\lambda(f\oplus g) = \bigoplus_{x \in X} \lambda(x) \odot (f(x)\oplus g(x)) =\] \[=\left( \bigoplus_{x \in X} \lambda(x) \odot f(x)\right)\oplus \left( \bigoplus_{x \in X} \lambda(x) \odot g(x)\right)=  m_\lambda( f)\oplus m_\lambda(g).\]

We claim that $\gamma$ is surjective. From Lemma~\ref{lem: represent delta dirac property} we know that for any $m\in C^{*}(X,\mathbb{R}) $ we have
$$m(h) = \bigoplus_{x \in X} F(h(x),x )= \bigoplus_{x \in X} F(0,x) \odot h(x).$$
Defining $\lambda(x):= F(0,x)$, which is u.s.c. from Lemma~\ref{lem: represent delta dirac property},  we obtain $m=m_{\lambda}$.

We claim that $\gamma$ is injective. Given $\lambda \in U(X, \mathbb{R}_{\max})$ and $m=m_\lambda$, let $F(0,x)$ as given above. We will prove that $\lambda(x) = F(0,x)$. Indeed, for a fixed $x_0$, consider the standard continuous approximation $(f_n)$ of $g^{0}_{x_0}$. As $\lambda $ is u.s.c. we have
\[\lambda(x_0) = \lim_{n\to\infty} \bigoplus_{x\in X} \lambda(x)\odot f_n(x) \]
Then
\[\lambda(x_0) = \lim_{n\to\infty} \bigoplus_{x\in X} \lambda(x)\odot f_n(x) = \lim_{n\to\infty} m_\lambda( f_n) =\tilde{m}(g^{0}_{x_0}) = F(0,x_0).\]

We need also to show that $\gamma$ is max-plus linear. To see that take,
$$\gamma(a \odot \lambda)(h)= \bigoplus_{x \in X} a \odot \lambda(x) \odot h(x)= a \odot \bigoplus_{x \in X}\lambda(x) \odot h(x)= a \odot \gamma(\lambda)(h),$$
and
$$\gamma(\lambda \oplus \lambda')(h)= \bigoplus_{x \in X} (\lambda(x) \oplus \lambda'(x)) \odot h(x)= $$ $$=\left(\bigoplus_{x \in X} \lambda(x) \odot h(x) \right) \oplus \left( \bigoplus_{x \in X} \lambda'(x) \odot h(x)\right)=\gamma(\lambda )(h) \oplus  \gamma(\lambda')(f),$$
for all $h \in C(X, \mathbb{R})$.

(2) We notice that $m_\lambda(0) = \bigoplus_{x\in X} \lambda(x)\odot 0(x) = \bigoplus_{x\in X} \lambda(x)$, thus  $\lambda$ is always bounded by $m_\lambda(0)$.
Furthermore $\bigoplus_{x \in X} \lambda(x)=0$ if and only if $m_{\lambda}(0) = 0$.

\end{proof}

\begin{remark}
	The formula $$m_{\lambda}(h):=\bigoplus_{x \in X} \lambda(x) \odot h(x),$$
	for any $h \in C(X,\mathbb{R})$, in Theorem~\ref{thm: charact idempotent measures} can be slightly improved by recalling that, for each $x \in X$ one can define the Dirac delta measure $\delta_{x}: C(X,\mathbb{R}) \to \mathbb{R}$ by the formula $\delta_{x}(h):= h(x).$  It is obvious that $\delta_{x} \in C^*(X,\mathbb{R})$ (the density of $\delta_{x}$ is $g^{0}_{x}$). Thus, we can see the previous formula as a kind of Choquet's theorem (see \cite[Theorem Choquet, Pg. 14]{Ph01}), where any idempotent measure is a max-plus combination of Dirac delta measures $\delta_{x}$ with coefficients $\lambda(x)$, that is, $m_{\lambda}(h)=\bigoplus_{x \in X} \lambda(x) \odot \delta_{x}(h),$ for any $f \in C(X,\mathbb{R})$ or,
	$$m_{\lambda}=\bigoplus_{x \in X} \lambda(x) \odot \delta_{x}.$$
\end{remark}

\subsection{Upper semicontinuous envelope}

\begin{definition}\label{def:bounded functions rho}
	Let $\mathcal{B}(X, \mathbb{R}_{\max})$ be the set of bounded functions in the sense of $\rho$, that is, $f: X \to \mathbb{R}_{\max}$ is bounded if there exists $K>0$ such that $d_{\rho}(f, -\infty)=\max_{x \in X}\rho(f(x), -\infty) \leq K$.
\end{definition}
We notice that $f$ is bounded with respect to $\rho$ if, and only if, it is bounded from above in the usual sense because
$\rho(f(x), -\infty) \leq K \Leftrightarrow \exp(f(x)) \leq K  \Leftrightarrow f(x) \leq \ln(K).$ As $X$ is compact any u.s.c. function (see Definition~\ref{def:usc}) is in  $\mathcal{B}(X, \mathbb{R}_{\max})$.

\begin{definition}\label{def: upper semicontinuous envelope}
	Given a function $\lambda \in \mathcal{B}(X, \mathbb{R}_{\max})$ we define the upper semicontinuous envelope of $\lambda$  as
	$$\lambda^{u.s.e.}(x)=\inf_{\varphi \in C(X, \mathbb{R}), \; \varphi \geq \lambda} \varphi(x), \; \forall x \in X.$$
\end{definition}
In \cite{KM89} is used the lower semicontinuous envelope. The main properties of the upper semicontinuous envelope are given in the next lemma.
\begin{proposition}\label{prop: upper semicontinuous envelope} By considering Definition \ref{def: upper semicontinuous envelope},
	\begin{enumerate}
		\item If $\lambda \in \mathcal{B}(X, \mathbb{R}_{\max})$ and $\lambda \neq -\infty$ then $\lambda^{u.s.e.}  \in U(X, \mathbb{R}_{\max})$;
		\item If $\lambda \in U(X, \mathbb{R}_{\max})$ then $\lambda^{u.s.e.}=\lambda$.
	\end{enumerate}
\end{proposition}
\begin{proof}
	(1) First, we notice that $\lambda \neq -\infty$ means that there exists $x_0 \in X$ such that $-\infty <\lambda(x_0)\leq \lambda^{u.s.e.}(x_0)$ thus, $\supp(\lambda^{u.s.e.}) \neq \varnothing$. We claim that $\lambda^{u.s.e.}$ is u.s.c. Indeed, given $x\in X$ and $\varepsilon>0$, let $(x_n)$ be any sequence such that $x_{n}\to x$. We fix any function $\varphi_0 \in C(X, \mathbb{R})$ such that $\varphi_0 \geq \lambda$. Then we have
	$$\lambda^{u.s.e.}(x_n)=\inf_{\varphi \in C(X, \mathbb{R}), \; \varphi \geq \lambda} \varphi(x_n) \leq \varphi_0(x_n)  \leq  \varphi_0(x)+\varepsilon $$
	for $n$ big enough. Thus,
	$$\limsup_{n\to \infty}\lambda^{u.s.e.}(x_n) \leq \varphi_0(x)+\varepsilon.$$
	If we take the infimum over functions $\varphi_0$ at the right hand side of this inequality, we get
	$$\limsup_{n\to \infty}\lambda^{u.s.e.}(x_n) \leq  \lambda^{u.s.e.}(x)+\varepsilon.$$
	As $\varepsilon$ is arbitrary we conclude that $\lambda^{u.s.e.}$ is u.s.c.. As $\supp( \lambda^{u.s.e.}) \neq \varnothing$ we have proved that $\lambda^{u.s.e.}  \in U(X, \mathbb{R}_{\max})$.\\
	(2) As $X$ is compact and $\lambda$ is u.s.c. we know that $\lambda$ attain its maximum value which we will denote by $M$. Let $x_0$ be a point of $X$. As $\lambda$ is u.s.c. at $x_0$, for any real $c$ satisfying $1+M>c> \lambda \left(x_0\right)$, there exists  $n\in\mathbb{N}$ such that $\lambda(x)< c$ for all $x \in B_{\frac{1}{n}}(x_0)$.
	Consider the continuous function $\varphi_{c}$ defined by
	\begin{equation}\label{varphi} \varphi_{c}(y):=
		\left\{
		\begin{array}{ll}
			1+M, & d(y,x_0)>1/n \\
			c+n(1+M-c)d(y,x_0) , &  d(y,x_0)\leq  1/n
		\end{array}
		\right. ,
	\end{equation}
	and observe that $\varphi_{c} \geq \lambda$. Thus, $\lambda(x_0) \leq \lambda^{u.s.e.} (x_0) \leq \varphi_c(x_0) =c.$ As $c>\lambda(x_0)$ is arbitrary we conclude that $\lambda^{u.s.e.}(x_0)=\lambda(x_{0})$.
\end{proof}

\begin{proposition}\label{prop: same envelope}
	For each $\lambda \in \mathcal{B}(X, \mathbb{R}_{\max}) $ consider the functional $m_{\lambda}: C(X,\mathbb{R}) \to \mathbb{R}_{\max}$ defined by
	$$m_{\lambda}(h):=\bigoplus_{x \in X} \lambda(x) \odot h(x),$$
	for any $h \in C(X,\mathbb{R})$. If $\lambda_{1}, \lambda_{2} \in \mathcal{B}(X, \mathbb{R}_{\max})$ and $m_{\lambda_{1}}= m_{\lambda_{2}},$ then $\lambda_{1}^{u.s.e.}= \lambda_{2}^{u.s.e.}$.
	\end{proposition}

\begin{proof} Since $m_{\lambda_{1}}= m_{\lambda_{2}}$ we have
	$$\bigoplus_{x \in X} \lambda_{1}(x) \odot h(x) = \bigoplus_{x \in X} \lambda_{2}(x) \odot h(x),$$
	for any $h \in C(X,\mathbb{R})$. Given $\varphi \in C(X, \mathbb{R})$ we have $\varphi \geq \lambda_{1}$ if and only if $\varphi \geq \lambda_{2}$. Indeed, suppose $\varphi\geq \lambda_1$ and take $h=-\varphi$. Then we get
	$$0 \geq \bigoplus_{x \in X} \lambda_{1}(x) -\varphi(x) = \bigoplus_{x \in X} \lambda_{2}(x) -\varphi(x),$$
	meaning that $\varphi \geq \lambda_{2}$. The reverse argument is analogous.\\
	Finally, fixed $x_0 \in X$, we have
	$$\lambda_{1}^{u.s.e.}(x_0)=\inf_{\varphi \in C(X, \mathbb{R}), \; \varphi \geq \lambda_{1}} \varphi(x_0) =\inf_{\varphi \in C(X, \mathbb{R}), \; \varphi \geq \lambda_{2}} \varphi(x_0) = \lambda_{2}^{u.s.e.}(x_0) .$$
\end{proof}

\subsection{Maximal extension of idempotent measures}

In this section we present a brief discussion of how the domain of an idempotent measure can be extended from $C(X,\mathbb{R})$ to $\mathcal{B}(X, \mathbb{R}_{\max})$.

\begin{proposition}
	Let $\bar{C}(X,\mathbb{R})$ and $\tilde{m}$ as defined in Lemma \ref{maximal} where $m=m_\lambda$. If $g \in \bar{C}(X,\mathbb{R})$ then
	\[\tilde{m}(g) = \bigoplus_{x\in X} \lambda(x)\odot g(x).\]	
\end{proposition} 	
\begin{proof}
Let $f_n$ be a nonincreasing sequence of continuous functions converging pointwise to $g$. Denote by
\[\tilde{m}(g) = \lim_{n\to \infty}  \bigoplus_{x\in X} \lambda(x)\odot f_n(x)\]
and
\[M_\lambda(g) = \bigoplus_{x\in X} \lambda(x)\odot g(x).\]
We want to prove that $\tilde{m}(g) = M_\lambda(g)$. As $f_n\geq g$ we have
\[\tilde{m}(g) = \lim_{n\to \infty}  \bigoplus_{x\in X} \lambda(x)\odot f_n(x) \geq \bigoplus_{x\in X} \lambda(x)\odot g(x)=M_\lambda(g).\]
In order to prove the equality suppose initially $M_\lambda(g)\neq -\infty$ and by contradiction suppose there exists an $\varepsilon>0$ and a sequence $(x_n)$ in $X$ such that  $\lambda(x_n)\odot f_n(x_n) > M_\lambda(g)+\varepsilon. $ We can suppose there exists $x_0\in X$ such that $x_n\to x_0$.

Case 1: supposing $\lambda(x_0)\odot g(x_0) \neq -\infty$. As $\lambda(x_0)\odot f_n(x_0)$ converges to $\lambda(x_0)\odot g(x_0)$  there exists $k$ such that $\lambda(x_0)\odot f_k(x_0)< \lambda(x_0)\odot g(x_0)+\varepsilon\leq M_{\lambda}(g)+\varepsilon$. As $\lambda+f_k$ is u.s.c. we get a $\delta>0$ such that
$$d(x,x_0)<\delta \Rightarrow \lambda(x)\odot f_k(x) < M_\lambda(g)+\varepsilon .$$ As $f_n$ is nonincreasing we get $$[d(x,x_0)<\delta, n\geq k] \Rightarrow \lambda(x)\odot f_n(x) < M_\lambda(g)+\varepsilon,$$
which is a contradiction because $d(x_n,x_0)<\delta$ for $n$ large enough.

Case 2: supposing $\lambda(x_0)\odot g(x_0) = -\infty$. Then, for any natural $N$ there exists $k$ such that $\lambda(x_0)\odot f_k(x_0)< -N$. As $\lambda+f_k$ is u.s.c. we get a $\delta>0$ such that $$d(x,x_0)<\delta \Rightarrow \lambda(x)\odot f_k(x) <-N.$$ As $f_n$ is nonincreasing we get $$[d(x,x_0)<\delta, n\geq k] \Rightarrow \lambda(x)\odot f_n(x) < -N,$$
which is a contradiction because $d(x_n,x_0)<\delta$ for $n$ large enough.

Now we suppose $M_\lambda(g)= -\infty$, which means $\lambda(x)+g(x) = -\infty,\;\forall \,x\in X$, and by contradiction suppose there exists a real number $L$ and a sequence $(x_n)$ in $X$ such that  $\lambda(x_n)\odot f_n(x_n) > L. $ We can suppose there exists $x_0\in X$ such that $x_n\to x_0$. Clearly $\lambda(x_0)\odot g(x_0) = -\infty$ and we get a contradiction arguing as in case 2 above.

\end{proof}

From Theorem~\ref{thm: charact idempotent measures} we have a formula for an idempotent measure $m_{\lambda}$, with density $\lambda$, which can be extended to $ h \in \mathcal{B}(X, \mathbb{R}_{\max})$ by setting up
$$\hat{m}_{\lambda}(h) = \bigoplus_{x \in X} \lambda(x) \odot h(x),$$
where the supremum is well defined because $h$ and $\lambda$ are bounded from above.  The functional $\hat{m}_{\lambda}$ is called the maximal extension of $m_{\lambda}$ (see \cite[Corollary 1]{KM89}, \cite[Theorem 2.2]{Aki98} and \cite[Pg. 36]{Kol97}, for a minimal version). The precise meaning of this name is explained by Lemma~\ref{maximal} and by Remark~\ref{rem: maximal}.

\begin{definition}\label{def: idemp integral}
  Given $\lambda \in \mathcal{B}(X, \mathbb{R}_{\max})$ we define the idempotent integral for $m_{\lambda} \in C^{*}(X,\mathbb{R})$ by the formula
  $$\int_{X} h(x) dm_{\lambda}(x):= \hat{m}_{\lambda}(h) = \bigoplus_{x \in X} \lambda(x) \odot h(x),$$
  for all $ h \in \mathcal{B}(X, \mathbb{R}_{\max})$. Obviously, if $h \in C(X, \mathbb{R})$ then $\int_{X} h(x) dm_{\lambda}(x)= m_{\lambda}(h)$.
\end{definition}

\begin{definition}\label{def: max-plus indicator}
  Given any $A \subset X$ we define the max-plus indicator function of $A$ as the function
  \[ \chi_{A}(x):=
  \left\{
    \begin{array}{ll}
      0, & x \in A \\
      -\infty, & x \not\in A
    \end{array}
  \right. .
  \]
\end{definition}

Obviously, $\chi_{A} \in \mathcal{B}(X, \mathbb{R}_{\max})$, for any $A \subset X$, so we can compute the integral of $\chi_{A}$ with respect to an idempotent measure $m_{\lambda}$ by applying the extension $\hat{m}_{\lambda}$
$$\int_{X} \chi_{A} d\mu:= \hat{m}_{\lambda}(\chi_{A}) = \bigoplus_{x \in X} \lambda(x) \odot \chi_{A}(x) = \bigoplus_{x \in A} \lambda(x).$$

We recall that $2^X$ is the set of parts of a set $X$.
\begin{definition}\label{def: set function idempotent measure}
  Consider $m_{\lambda} \in C^{*}(X,\mathbb{R})$ with extension $\hat{m}_{\lambda}$. The correspondence (we use the same $m_{\lambda}$ as symbol if there is no risk of confusion) $m_{\lambda}: 2^X  \to \mathbb{R}$ defined by
  $$m_{\lambda}(A):= \int_{X} \chi_{A}(x) dm_{\lambda}(x) = \bigoplus_{x \in A} \lambda(x),$$
for all $A \in 2^X $ is called the \emph{set idempotent measure} (see \cite[Corollary 1]{Kol97}, \cite[Definition 2.1]{Aki98} for cost measures, \cite[Section 3]{Aki99} and \cite[Definition 1]{MoDo98}).
\end{definition}

\begin{remark}\label{rem: maximal} By item (3) of Theorem~\ref{thm: charact idempotent measures},  if $\lambda \in U(X, \mathbb{R}_{\max})$ and $\lambda_1 \in \mathcal{B}(X, \mathbb{R}_{\max})$ satisfy  $m_{\lambda_{1}}= m_{\lambda}$ then  $\lambda= \lambda_{1}^{u.s.e.} \geq \lambda_{1}$. Thus,
$$m_{\lambda}(A)= \bigoplus_{x \in A} \lambda(x) \geq \bigoplus_{x \in A} \lambda_{1}(x)=m_{\lambda_{1}}(A),$$
for all $A \in 2^X$. So, the choice of the u.s.c. function $\lambda$ produces an integral bigger or equal to any other choice $\lambda_1$, which explains the meaning of the name ``maximal".
\end{remark}

\begin{proposition}\label{prop: prop set idemp measure}
  Let $m_{\lambda} \in C^{*}(X,\mathbb{R})$ with extension $\hat{m}_{\lambda}$. Then,
\begin{enumerate}
  \item $m_{\lambda}(\varnothing)=-\infty$;
  \item $m_{\lambda}(A \cup B)= m_{\lambda}(A ) \oplus m_{\lambda}(B)$ for any $A,B \in 2^X$.
\end{enumerate}
\end{proposition}
\begin{proof}
(1)  As usual in the set function theory, we assume that, for any function $f: X \to \mathbb{R}_{\max}$ the supremum over an empty set is the smallest value  in $\mathbb{R}_{\max}$, that is, $-\infty$. Thus,  $m_{\lambda}(\varnothing)=\bigoplus_{x \in \varnothing} \lambda(x) = -\infty$.

(2) $$m_{\lambda}(A \cup B)=\int_{X} \chi_{_{A \cup B}}(x) dm_{\lambda}(x)= \hat{m}_{\lambda}(\chi_{_{A \cup B}}) = \bigoplus_{x \in X} \lambda(x) \odot \chi_{_{A \cup B}}(x) =$$
$$=\bigoplus_{x \in X} \lambda(x) \odot ( \chi_{_{A}}(x) \oplus \chi_{_{B}}(x)) = \bigoplus_{x \in B} \lambda(x)  \oplus
\bigoplus_{x \in A} \lambda(x) = m_{\lambda}(A ) \oplus m_{\lambda}(B).$$
\end{proof}

\end{document}